\documentclass{amsart}
\usepackage{amscd}
\usepackage{amsmath}
\usepackage{diagrams}
\usepackage{enumerate}
\usepackage{todonotes}
\usepackage{cite}
\usepackage{graphicx}
\usepackage{amssymb}
\usepackage{amsxtra}
\usepackage{amstext}
\usepackage{amsbsy}
\usepackage{amsthm,latexsym,calligra,mathrsfs}
\usepackage{amsfonts,euscript,delarray,enumitem}
\usepackage{fancyhdr,nameref,amscd}
\usepackage[all,cmtip]{xy}
\usepackage[letter,center]{crop}
\usepackage{geometry}
\usepackage{tikz,tikz-cd}
\usetikzlibrary{arrows.meta}
\tikzcdset{every label/.append style = {font = \small}}
\usetikzlibrary{matrix,arrows}

\newtheorem{theorem}{Theorem}[section]
\newtheorem{lemma}[theorem]{Lemma}
\newtheorem{proposition}[theorem]{Proposition}
\newtheorem{corollary}[theorem]{Corollary}
\newtheorem{definition}[theorem]{Definition}

\newtheorem{remark}[theorem]{Remark}

\newtheorem{example}[theorem]{Example}

 \newtheorem{theorem*}{Theorem}
 \newtheorem{corollary*}[theorem*]{Corollary}
 \newtheorem{proposition*}[theorem*]{Proposition}

\newcommand{\CC}{{\mathbb{C}}}

\newcommand{\PP}{{\mathbb{P}}}
\newcommand{\QQ}{{\mathbb{Q}}}

\newcommand{\ZZ}{{\mathbb{Z}}}

\newcommand{\smvee}{\raise0.9ex\hbox{$\scriptscriptstyle\vee$}}

\newcommand{\Ss}[1]{\mathcal{O}_{#1}}

\newcommand{\Rs}{\tilde{X}}

\newcommand{\Sf}[1]{\mathcal{#1}}

\DeclareMathOperator{\Exts}{\mathscr{E}\text{\kern -3pt {\calligra\large xt}}\,}
\DeclareMathOperator{\Homs}{\mathscr{H}\text{\kern -3pt {\calligra\large om}}\,}
\DeclareMathOperator{\Hom}{\text{Hom}\,}

\begin{document}

 \title[The classification of reflexive modules of rank one]{The classification of reflexive modules of rank one over rational and minimally elliptic singularities}

\author{Andr\'as N\'emethi}
\address{Alfr\'ed R\'enyi Institute of Mathematics,
Re\'altanoda utca 13-15, H-1053, Budapest, Hungary \newline
 \hspace*{4mm} ELTE - University of Budapest, Dept. of Geo.,
  P\'azm\'any P\'eter s\'et\'any 1/A, 1117 Budapest, Hungary
  \newline \hspace*{4mm}
  BBU - Babe\c{s}-Bolyai Univ., Str, M. Kog\u{a}lniceanu 1, 400084 Cluj-Napoca, Romania
   \newline \hspace*{4mm}
BCAM - Basque Center for Applied Math.,
Mazarredo, 14 E48009 Bilbao, Basque Country – Spain}
\email{nemethi.andras@renyi.hu }

\author{Agust\'in Romano-Vel\'azquez}

\address{Alfr\'ed R\'enyi Institute Of Mathematics,  Re\'altanoda utca 13-15, H-1053, Budapest, Hungary \newline
 \hspace*{4mm} Instituto de Matem\'aticas, Unidad Cuernavaca\\ Universidad Nacional Aut\'onoma de M\'exico, \newline \hspace*{4mm} Avenida Universidad s/n, Colonia Lomas de Chamilpa CP62210, Cuernavaca, Morelos Mexico }
\email{agustin@renyi.hu, agustin.romano@im.unam.mx}

\thanks{The authors are partially supported by NKFIH Grants ``\'Elvonal (Frontier)'' KKP 126683 and KKP 144148.}



\date{\today}

\begin{abstract}
We classify the  reflexive modules of rank one over rational and minimally elliptic singularities. Equivalently,
we classify full line bundles on the resolutions of rational and minimally elliptic singularities.
As an application, we determine among such  reflexive modules of rank one all the special ones  (in the sense of Wunram) and
all the flat ones. In this way, we also classify the non-flat reflexive modules as well
(as a generalization of a construction of Dan and Romano).
In particular, we prove (in the rank one case) a  conjecture of Behnke, namely that
in the case of a cusp singularity any reflexive module admits a flat connection.

The results generalize the classical  Mckay correspondence, and results of Artin, Verdier, Esnault, Khan and Wunram
valid for different particular families of singularities.
\end{abstract}

 \maketitle

 \section{Introduction}
Let $(X,x)$ be normal surface singularity, embedded in $(\CC^n,0)$. Denote by $\Sigma$ its link, i.e., $\Sigma=X \cap \mathbb{S}_\epsilon^{2n-1}$ with $\epsilon>0$ small enough. The first complete classification of the finite dimensional representations of $\pi_1(\Sigma)$ was done by McKay~\cite{McK} in the case of rational double point singularities.
This provides an identification of the McKay graph of the non-trivial irreducible representations with the
dual resolution graph of the minimal resolution.
This is called the McKay correspondence. Later Artin and Verdier~\cite{AV}, reformulate the McKay correspondence in a more geometrical setting. For rational double singularities the McKay correspondence by Artin and Verdier gives a complete classification of the indecomposable reflexive $\Ss{X}$-modules. The classification of reflexive modules has been studied by several people in different cases: Esnault~\cite{Esnault} introduced full sheaves and
 improved the results of Artin and Verdier~\cite{AV} for  rational surface singularity,
 and classified rank one reflexive modules for quotient singularities.
 Wunram~\cite{Wu} classified the family of \emph{special} reflexive modules for rational surface singularities, nevertheless the complete classification of reflexive modules on rational singularities remained open. For a non-rational singularity, Khan~\cite{Ka} classified all the reflexive modules in the simply elliptic case using the Atiyah’s classification of vector bundles on elliptic curves. For a general normal Gorenstein singularity, Bobadilla and Romano~\cite{BoRo} classified the family of \emph{cohomological special} reflexive modules. For a general surface singularity, the classification problem for reflexive $\Ss{X}$-modules remains open.

In this article, we provide a complete classification of reflexive modules of rank one for rational surfaces singularities and minimally elliptic singularities. Let $\pi \colon \Rs \to X$ be a resolution of singularities. Denote the exceptional divisor by $E:=\pi^{-1}(x)$. Set $L:=H_2(\Rs,\ZZ)$ and denote by $L'$ its dual. Set $H:=\mathrm{Tors}  ( H_1(\partial \Rs,\ZZ) )= \mathrm{Tors}( H_1(\Sigma,\ZZ) )\cong L'/L$. We set $[\ell']$ the class of $\ell' \in L'$ in $H$. Let $\mathcal{S}'$ be the Lipman (anti-nef) cone. For any $h \in H$, there exists a unique $s_h \in \mathcal{S}'$ such that $[s_h]=h$ and $s_h$ is minimal (see Section~\ref{Sec:Pre} for more details). Note that if $(X,x)$ is a rational singularity, then the first Chern class
$c_1:\mathrm{Pic}(\Rs)\to H^2(\Rs,\ZZ)\simeq L'$
realizes  a bijection between $\mathrm{Pic}(\Rs)$ and $L'$. In particular, for any $s_h \in L'$ there exists a unique line bundle denoted by $\Ss{X}(s_h)$ such that $c_1(\Ss{X}(s_h))=s_h$.

 Now, following Esnault~\cite{Esnault} and Khan~\cite{Ka} we say that a locally free sheaf $\Sf{M}$ over $\Rs$ is \emph{full} if and only if $\Sf{M}\cong (\pi^*M)^{\smvee \smvee}$ with $M$ a reflexive $\Ss{X}$-module. Our first main result (Theorem~\ref{th:class.rational.case}) is the following.
\begin{theorem*}
 Let $(X,x)$ be a rational normal surface singularity and
 let $\pi\colon \Rs \to X$ be any resolution
 of $(X,x)$. Then
 \begin{enumerate}
     \item for any $h\in H$ the line bundle $\Ss{\Rs}(-s_h)$ is full.
     \item If a line bundle $\mathcal{L}\in \mathrm{Pic}(\Rs)$ is full, then $c_1(\mathcal{L})=-s_h$ for some $h\in H$. 
 \end{enumerate}
 Thus, the set of full sheaves  of rank one is exactly
  $\{\Ss{\Rs}(-s_h)\}_{h\in H}$.
\end{theorem*}
We wish to emphasize  that in this  result we do not impose any speciality condition (compare with Wunram~\cite{Wu} and with the discussion below).

The above  classification generalizes the McKay correspondence (for rank one) in the case of rational surface singularities in the following sense.
In Theorem \ref{theorem:special.rat} and Corollary \ref{theorem:special.rat2}
we classify all rank one {\it special} reflexive modules (full sheaves) supported on
rational singularities, as part of reflexive modules (full sheaves). Their classification is
indeed guided by certain irreducible exceptional curves (as in the case of McKay correspondence).

If the singularity is not rational, then the first Chern class does  not provide an
 isomorphism anymore,
hence the identification of certain line bundles (with special properties) from the class of line bundles with given Chern class is considerably harder.

We denote by $\mathrm{Pic}^{\ell'}(\Rs):= c_1^{-1}(\ell')$ for any $\ell'\in L'$.
It is an affine space, a torsor of $\mathrm{Pic}^{0}(\Rs)\simeq \CC^{p_g}$ (where $p_g$ denotes the geometric genus).
 We also denote by $\mathrm{Full}^1(\Rs)$ the set of  full sheaves over $\Rs$ of rank one.
 We regard  $\mathrm{Full}^1(\Rs)$  as a subset of $\mathrm{Pic}(\Rs)$.

 Our second main result (Theorem~\ref{theo:class.general}) is the following.
\begin{theorem*}
Let $(X,x)$ be a minimally elliptic singularity. Let $\pi \colon \Rs \to X$ be any resolution such that the support of the elliptic cycle is  equal to  $E$. (This happens e.g. in the minimal resolution.) Then
\begin{equation*}
    \mathrm{Full}^1(\Rs) \subset \left( \bigcup_{h\in H\setminus\{0\}}\mathrm{Pic}^{-s_h}(\Rs) \right) \cup \left ( \mathrm{Pic}^{-Z_{min}}(\Rs) \setminus \{\Ss{\Rs}(-Z_{min})\} \right) \cup \{\Ss{\Rs}\}.
\end{equation*}
Moreover, suppose that any irreducible component of $E$ is isomorphic to $\mathbb{P}_{\CC}^1$. Then,
\begin{equation*}
    \mathrm{Full}^1(\Rs) = \left( \bigcup_{h\in H\setminus\{0\}}\mathrm{Pic}^{-s_h}(\Rs) \right) \cup \left ( \mathrm{Pic}^{-Z_{min}}(\Rs) \setminus \{\Ss{\Rs}(-Z_{min})\} \right)\cup \{\Ss{\Rs}\}.
\end{equation*}
\end{theorem*}
Again, note that this classification theorem does not require  the `specialty' property.

Next, we analyze the existence of flat connections in the reflexive modules. Recall that
 the original McKay correspondence establishes a one-to-one correspondence between finite dimensional non-trivial irreducible representations of $\pi_1(\Sigma)$ and the irreducible components of the exceptional divisor of the minimal resolution. By the Riemann-Hilbert correspondence (see~\cite{Gustavsen1}) there is a one-to-one correspondence between finite dimensional representations of $\pi_1(\Sigma)$ and reflexive $\Ss{X}$-modules equipped with a flat connection. Denote by
\begin{align*}
    \mathrm{Ref}^1_{X} &:= \text{the category of reflexive $\Ss{X}$-modules of rank one},\\
    \mathrm{Ref}^{1,\nabla}_{X} &:= \text{the category of pairs $(M,\nabla)$ where $M$ is a reflexive $\Ss{X}$-module of rank one}\\
        &\,\quad \,\text{and $\nabla$ is an integrable connection}.
\end{align*}
In general, the forgetful functor
\begin{equation*}
    \mathrm{Ref}^{1,\nabla}_{X} \to \mathrm{Ref}^1_{X},
\end{equation*}
is not an equivalence of categories. Moreover, the forgetful functor may not be essentially surjective, i.e., the forgetful functor may not be onto on objects. In the case of quotient singularities, Esnault~\cite{Esnault} proved that the forgetful functor is essentially surjective. In the case of simple elliptic singularities, Kahn~\cite{Ka} also proved that the forgetful functor is essentially surjective.

In Theorem \ref{th:new} we prove that in the case of rational singularities {\it any rank one
full sheaf is flat}, i.e. the forgetful functor is essentially surjective.

The first example of a singularity such that the forgetful functor is not essentially surjective was done by Dan and Romano~\cite{DanRom}. As our first application (Theorem~\ref{theo:class.1}), we generalize the results and ideas  of Dan and Romano as follows:
\begin{theorem*}
Let $(X,x)$ be a minimally elliptic singularity such that its link is a rational homology sphere. Let $\pi \colon \Rs \to X$ be the minimal resolution. Corresponding to $h=[c_1({\mathcal L})]=0$ there is only one flat full sheaf, namely the trivial one. Let $h\in H$ be different from zero. Then, again,
in the family $\mathrm{Pic}^{-s_h}(\Rs)\simeq\CC$ of full sheaves
only one element  admits a flat connection
(which is concretely characterized).
\end{theorem*}
In~\cite{Be} Behnke conjectured that every reflexive module on a cusp singularity admits a flat connection. Later Gustavsen and Ile~\cite{Gustavsen2} extended  this conjecture for any log-canonical surface singularity. Our last application (Theorem~\ref{theo:class.2}) is a partial positive answer to the  conjecture (valid for rank one modules).
\begin{theorem*}
Let $(X,x)$ be a cusp singularity. Then, any reflexive module of rank one admits a flat connection.
\end{theorem*}
The paper is organized as follows. In Section~\ref{Sec:Pre} we review some general properties of normal surface singularities, the divisor lattice structure and the Lipman cone, flat reflexive modules and full shaves. In Section~\ref{sec:rational.sing} we classify the rank one full sheaves for rational surfaces singularities using techniques based on the properties of the
 divisor lattice structure and of the Lipman cone. In Section~\ref{sec:min.ell} we generalize the techniques of the previous section to the case of minimally elliptic singularities.  In Section~\ref{sec:flat.noflat} we use our classification theorems to study the flat reflexive modules over minimally elliptic singularities.

\section{Preliminaries}
\label{Sec:Pre}
In this section we recall basics on reflexive modules, full sheaves and topological invariants of normal surface singularities. We assume basic familiarity with these objects, see~\cite{BrHe,Ne,Ishi,NeBook} for more details.

\subsection{Setting and notation. Normal singularities.} Throughout this article, we denote by $(X,x)$ the germ of a complex  analytic complex normal surface singularity, i.e., the germ of a complex surface such that its local ring of functions $\Ss{X,x}$ is integrally closed in its field of fractions. In this situation $X$ has a \emph{dualizing sheaf} $\omega_X$, and let $\omega_{X,x}$ be its stalk at $x\in X$, which is called the \emph{dualizing module} of the ring $\Ss{X,x}$ (see~\cite[Chapter~5~\S~3]{Ishi} for more details). We say that $(X,x)$ has a  \emph{Gorenstein} normal singularity if the dualizing module is isomorphic to $\Ss{X,x}$.

\subsection{Good resolutions and dual graphs}
Let $(X,x)$ be the germ of an analytic complex normal surface singularity. Let
\begin{equation*}
    \pi \colon \Rs\to X,
\end{equation*}
be a \emph{resolution of} $(X,x)$, i.e., a proper holomorphic map from a smooth surface $\Rs$ to a given representative of $(X,x)$ such that $\pi$ restricted to the complement of
 $\pi^{-1}(x)$ is biholomorphic. Sometimes we will require $\pi$ to be a \emph{good resolution}, which means that the exceptional divisor  $E:=\pi^{-1}(x)$ is a normal crossing divisor and each irreducible component of
$E$ is smooth. For any normal surface singularity there is always a good resolution, however it is
not unique. Given a good resolution, the \emph{dual graph} is a decorated graph $\Gamma$ constructed
as follows: the set of vertices, say $V$, is in bijection with the set of irreducible components of
$E$, say $\{E_v\}_{v\in V}$, two vertices $u$ and $v$ are connected by $k$ edges in $\Gamma$ if and
only if the corresponding components $E_u$ and $E_v$ have $k$ intersection points. Each vertex $v$ is decorated with two numbers: the self-intersection $E^2_v$ in $\Rs$ and the genus $g(E_v)$.

The \emph{intersection matrix} $M=(m_{u,v})_{u,v \in V}$ associated to the dual graph $\Gamma$ is the intersection matrix of the curves $\{E_v\}_{v\in V}$, i.e., $m_{u,v}=(E_u,E_v)$.
It is negative definite.

\subsection{The link}
Let us embed  $(X,x)$ in a certain  $(\CC^n,0)$. Via this embedding,
the \emph{link} $\Sigma$ of $(X,x)$ is defined as the intersection $\Sigma=X \cap \mathbb{S}_{\epsilon}^{2n-1}$,
where $\mathbb{S}_{\epsilon}^{2n-1}=\{z\in \CC^n\,:\, |z|=\epsilon\}$ and
 $\epsilon>0$ is small enough. The diffeomorphism type of the link does not depend on the embedding or on  $0<\epsilon\ll 1$. 

\subsection{The divisor lattice structure and the Lipman cone}
Let $(X,x)$ be the germ of an analytic complex normal surface singularity. Let $\pi \colon \Rs \to X$ be any resolution. Set $L:= H_2(\Rs,\ZZ)$, i.e., $L$ is the free abelian group generated by the classes $\{E_v\}_{v\in V}$. Denote also $L':=H_2(\Rs, \partial \Rs, \ZZ)$. Notice that $L'$ is the dual of $L$,  indeed by the Lefschetz–Poincar\'e duality we have
 the perfect pairing $H_2(\Rs,\ZZ) \otimes H_2(\Rs,\partial \Rs,\ZZ) \to \ZZ$, hence
\begin{equation*}
\Hom_{\ZZ}(L,\ZZ)\cong H_2(\Rs, \partial \Rs, \ZZ)=L'.
\end{equation*}
Since the intersection matrix $M$ is non-degenerate, the homological map $L \to L'$ is injective. Thus, by the homological long exact sequence of the pair $(\Rs,\partial \Rs)$ (where
$\partial \Rs\simeq \Sigma$) we have
\begin{equation}
\label{eq:suc.L.L'}
    0 \to L \to L' \to H_1(\partial \Rs,\ZZ) \to H_1(\Rs,\ZZ) \to 0.
\end{equation}
Since $H_1(\Rs,\ZZ)$ is free,  $H_1(\partial \Rs,\ZZ) \cong \mathrm{Tors} ( H_1(\partial \Rs,\ZZ) ) \oplus H_1(\Rs,\ZZ)$,  and the quotient $L'/L$ is identified with
\begin{equation*}
H:= \mathrm{Tors} \big( H_1(\partial \Rs,\ZZ) \big)= \mathrm{Tors} \left( H_1(\Sigma,\ZZ) \right).
\end{equation*}
We extend the intersection form $(,)$ from $L$ to $L_{\QQ}:=L \otimes \QQ$, which we denote by $(,)_\QQ$. Via $(,)_\QQ$ the lattice $L' \cong \Hom_{\ZZ}(L,\ZZ)$
can be identified with a lattice of rational cycles
\begin{equation*}
\{\ell'\in L_\QQ \, : \, (\ell',\ell)_\QQ \in \ZZ \text{ for any $\ell \in L$ }\} \subset L_\QQ.
\end{equation*}
Hence, we can identify $L'$ with $\bigoplus_{v\in V} \ZZ\langle E_v^*\rangle$, the lattice generated by the rational cycles $E_v^*\in  L_\QQ$ $(v\in V)$ defined via
\begin{equation*}
 (E_u^*,E_v)_\QQ = -\delta_{u,v} \, \text{(Kronecker delta) for any $u,v \in V$.}
\end{equation*}
Let $\ell'_1, \ell'_2  \in L_{\QQ}$ where $\ell'_j= \sum_v l_{jv}'E_v$ for $j\in \{1,2\}$. We consider the partial order in $L_{\QQ}$ giving by $\ell'_1 \geq \ell'_2$ if and only if $l_{1v}' \geq l_{2v}'$ for all $v\in V$. We set $\min\{\ell'_1, \ell'_2\}:=\sum_v \min\{l_{1v}',l_{2v}'\}E_v$.

Let $\ell' \in L'$. We denote its class in $H$ by $[\ell']$. The lattice $L'$ has a partition parametrized by $H$, where for any $h \in H$ we have
\begin{equation*}
    L_h':=\{\ell' \in L' \, | \, [\ell']=h\}.
\end{equation*}
Note that $L_0'=L$. Given any $h\in H$ we define $r_h:=\sum_v l_v' E_v$ as the unique element of $L_h'$ such that $0 \leq l_v' <1$. We define the \emph{rational Lipman cone} by
\begin{equation*}
    \mathcal{S}_{\QQ}:= \{\ell'\in L_{\QQ} \, | \, (\ell',E_v) \leq 0 \text{ for any $v\in V$}\}.
\end{equation*}
It is a cone generated over $\QQ_{\geq 0}$ by $E_v^*$. We set
\begin{equation*}
    \mathcal{S}':= \mathcal{S}_{\QQ} \cap L' \quad \text{and} \quad \mathcal{S}:= \mathcal{S}_{\QQ} \cap L.
\end{equation*}
Note that $\mathcal{S}'$ is the monoid  of anti-nef rational cycles of $L'$, it is generated over $\ZZ_{\geq 0}$ by the cycles $E^*_v$. The Lipman cone $\mathcal{S}'$ also has a natural equivariant partition indexed by $H$. We denote $\mathcal{S}_h'=\mathcal{S}' \cap L_h'$. The monoid $\mathcal{S}=\mathcal{S}_0'$ has the following properties:
\begin{enumerate}
    \item if $Z= \sum n_{v}E_{v} \in \mathcal{S}$ and $Z\neq 0$, then $n_v >0$ for all $v \in V$,
    \item if $Z_1,Z_2 \in \mathcal{S}$, then $Z_1+Z_2 \in \mathcal{S}$,
    \item if $Z_1,Z_2 \in \mathcal{S}$, then $\min\{Z_1,Z_2\}  \in \mathcal{S}$.
\end{enumerate}
Thus, $\mathcal{S} \setminus \{0\}$ has a unique minimal element $Z_{min}$ called the \emph{Artin's fundamental cycle (minimal cycle or numerical cycle)}. Similarly, for any $h \in H$, the set  $\mathcal{S}_h'$ has the following properties \cite{Ne,Ne1}:
\begin{enumerate}
    \item If $s_1,s_2 \in \mathcal{S}_h'$, then $s_1-s_2 \in L$ and $\min\{s_1,s_2\}\in \mathcal{S}_h'$,
    \item for any $h$ there exists a unique \emph{minimal cycle} $s_h:=\min \{\mathcal{S}_h'\}$,
    \item the set $\mathcal{S}_h'$ is a cone with \emph{vertex} $s_h$ in the sense that $\mathcal{S}_h'= s_h + \mathcal{S}_0'$.
\end{enumerate}

\begin{remark}
\label{remark:1} Clearly  $s_0=0$. In fact,
 $s_h \neq 0$ if and only if $h\neq0$.
\end{remark}
From now on, we will denote by $\lfloor\cdot \rfloor$ the integral part, that is
 $   \big\lfloor \sum_v r_vE_v  \big \rfloor :=  \sum_v \lfloor r_v \rfloor E_v.$

\begin{definition}
The {\em geometric genus} $p_g$ of $X$ is  the $\mathbb{C}$-vector space   dimension of $H^1(\Rs,\Ss{\Rs})$. It is independent of the choice of the resolution.
We say that $(X,x)$  has a \emph{rational singularity} if $p_g=0$.
\end{definition}
We denote by  $\Omega^2_{\Rs}$  the sheaf of holomorphic $2$-forms. The divisor of any meromorphic section of $\Omega^2_{\Rs}$ is called \emph{the canonical divisor};
it is  denoted by $K_{\Rs}$. The rational cycle $Z_K\in L'$ (supported on $E$)
that satisfies $(Z_K , E_v)_{\QQ} = -K_{\Rs} \cdot E_v$ for any $v \in V$ is called \emph{the canonical cycle}.
(It is independent of the choice of $K_{\Rs}$.)
 In the case of Gorenstein singularities, one proves that $Z_K\in L$,  see~\cite{Ne}.

 We define the Riemann-Roch expression
 \begin{equation*}
  \chi:L\to \ZZ, \text{ } \chi(\ell)=-(\ell, \ell-Z_K)_{\QQ}/2.
 \end{equation*}
By Artin's criterion \cite{Artin66}
$(X,x)$ is rational if and only if $\chi(\ell)>0$ for any $\ell>0$ $(\ell\in L)$.
For rational graphs (graphs satisfying Artin's criterion) one also has
$\chi(\ell)\geq 0$ for any $\ell\in L$. Furthermore, rational graphs are trees of $\PP^1$'s.

The next level of complexity of graphs is realized by elliptic graphs.

\begin{definition} The germ $(X,x)$ is called elliptic if $(X,x)$ is not rational but $\chi(l)\geq 0$ for any $l\in L$. By an equivalent definition, $(X,x)$ is elliptic if $\chi(Z_{min})=0$. (Cf. \cite{Ne,NeBook}.)
Both  topological criterions are  independent of the choice of the resolution.
\end{definition}
\begin{definition}
Let $(X,x)$ be an elliptic  normal surface singularity. Let $\pi \colon \Rs \to X$ be a resolution. A non-zero effective cycle $C$ supported on $E$ is called \emph{minimally elliptic cycle }
if $\chi(C)=0$ and for any $0 < l < C$, one has $\chi(l)>0$.
\end{definition}
Once a resolution is fixed, a minimally elliptic cycle always exists and it is unique.
\begin{definition}
We say that $(X,x)$ is a \emph{minimally elliptic singularity}, if the geometric genus of $(X,x)$ is one and $(X,x)$ is Gorenstein.
\end{definition}
Laufer proved the following topological characterization:  an elliptic  singularity is  minimally elliptic if and only if in the minimal resolution
 $C=Z_K = Z_{min}$, see~ \cite{Laufer77} or
 \cite[Theorem~7.2.15]{NeBook}.

 \vspace{1mm}

 As usual,  denote $H^1(\Rs,\Ss{\Rs}^*)$ by  $\mathrm{Pic}(\Rs)$.
Let  $c_1 \colon \mathrm{Pic}(\Rs) \to  H^2(\Rs,\ZZ) \cong  H_2(\Rs,\partial \Rs,\ZZ) =L'$
 denote the  `first Chern class morphism', and set
\begin{equation*}
    \mathrm{Pic}^{\ell'}(\Rs):= c_1^{-1}(\ell') \quad \quad \text{for any $\ell'\in L'$.}
\end{equation*}
It is a   $ \mathrm{Pic}^{0}(\Rs)\simeq \CC^{p_g}$ torsor.

The following vanishing theorems will be used several times.

\begin{theorem}[Generalized Grauert–Riemenschneider vanishing~{\cite[Theorem~6.4.3]{NeBook}}]
\label{Th:GGR}
 Let $(X,x)$ be a normal surface singularity. Let $\pi \colon \Rs \to X$ be any resolution. Let $\Sf{L}\in \mathrm{Pic}(\Rs)$
 such that $c_1(\Sf{L}(-K_{\Rs}))\in \Delta-\mathcal{S}_{\QQ}$ for some $\Delta \in L\otimes\QQ$ such that $\lfloor \Delta \rfloor = 0$. Then for any $\ell \in L_{>0}$ one has the vanishing $h^1(\ell,\Sf{L}|_{\ell})=0$. In particular, $h^1(\Rs,\Sf{L})=0$.
\end{theorem}

\begin{theorem}[Lipman's vanishing~{\cite[Theorem~11.1]{Lip}}]
\label{Th:Lipman.Vanish}
 Let $(X,x)$ be a rational normal surface singularity. Let $\pi \colon \Rs \to X$ be any resolution. If $\Sf{L}\in \mathrm{Pic}(\Rs)$ with $-c_1(\Sf{L})\in \mathcal{S}'$, then $h^1(\Rs,\Sf{L})=0$.
\end{theorem}

\subsection{Generalized Laufer's algorithm}\cite[Lemma~7.4]{Ne1}\label{ss:GenLau}
For any $\ell'\in L'$, there exists a unique minimal element $s(\ell') \in \mathcal{S}'$ such that $s(\ell') \geq \ell'$ and $s(\ell')-\ell' \in L$. Furthermore, $s(\ell')$ can be constructed by a computation sequence $\{x_i\}_{i=1}^t$ as follows:
\begin{enumerate}
\item set $x_0:=\ell'$,
    \item if $x_i$ is already constructed and $x_i \not \in \mathcal{S}'$, then there exists some $E_{v_i}$ such that $(x_i, E_{v_i})>0$. Then define $x_{i+1}:=x_i+E_{v_i}$ and repeat the algorithm.
\end{enumerate}
The procedure stops after finitely many steps, and the last term $x_t$ is $s(\ell')$.

Laufer's algorithm will be used in several different situations. E.g., one might consider a line bundle $\mathcal{L}\in {\rm Pic}^{-c}(\Rs)$, and then run the algorithm starting with
$c$ and ending with $s(c)$: $x_0=c$,  $x_{i+1}=x_i+E_{v_i}$, $(x_i, E_{v_i})>0$. Then in the cohomological long exact sequence associated with
\begin{equation}
\label{eq:exact.seq.lemma.comp.seq2}
   0 \to \mathcal{L}(c-x_{i+1}) \to \mathcal{L}(c-x_{i}) \to
   \mathcal{L}(c-x_{i})|_{E_{v_i}}\to 0,
\end{equation}
the Chern class of  $\mathcal{L}(c-x_{i})|_{E_{v_i}}$ is negative, so
$H^0(E_{v_i}, \mathcal{L}(c-x_{i})|_{E_{v_i}})=0$.
Hence we have the following statements proved inductively along the computation sequence.

\begin{lemma}
\label{lemma:computation.seq}
Let $(X,x)$ be the germ of a normal surface singularity. Let $\pi \colon \Rs  \to X$ be a resolution and fix  $\Sf{L} \in \mathrm{Pic}^{-c}(\Rs)$.  Then

(a) the natural map
$H^0(\Rs,\Sf{L}(c-s(c))) \to H^0(\Rs,\Sf{L})$ is an isomorphism,

(b) $h^1(\Rs, \Sf{L})=h^1(\Rs, \Sf{L}(c-s(c))+\sum_i h^1(E_{v_i},  \Sf{L}(c-x_i))$.

In particular, if $(X,x)$ is rational then $h^1(\Rs, \Sf{L}(c-s(c))=0$ by Lipman's vanishing, hence $h^1(\Rs, \Sf{L})=\sum_i h^1({\mathcal O}_{{\mathbb P}^1}(  -(x_i, E_{v_i})))$.
\end{lemma}
The last sentence shows that in rational case, $h^1(\Rs,\Sf{L})=0$ if and only if along the
computation sequence connecting $c$ and $s(c)$ at all the steps $(x_i, E_{v_i})=1$ holds.

This can be compared with Laufer's criterion of rationality \cite{Laufer72}. For any singularity
$(X,x)$ with resolution $\Rs$ (and all $E_v$ rational) consider the computation sequence starting
with one of the exceptional curves, say $E_{v_0}$,  and ending with $s(E_{v_0})=Z_{min}$.
Then $(X,x)$ is rational if and only if $h^1(\mathcal{O}_{Z_{min}})=0$ if and only
if along the sequence $(x_i,E_{v_i})=1$.

Part {\it (a)} of the above theorem will be used e.g. when $c=r_h$ and $s(r_h)=s_h$
 for some $h\in H$.



\subsection{Flat and reflexive modules}
Let $X$ be a normal variety. Let $\Homs_{\Ss{X}}(\cdot,\cdot)$
be the sheaf theoretical  Hom
functor.
The dual of an $\Ss{X}$-module $M$ is denoted by $M^{\smvee}:=\Homs_{\Ss{X}}(M,\Ss{X})$. An $\Ss{X}$-module $M$ is called \emph{reflexive} if the natural homomorphism from $M$ to $M^{\smvee \smvee}$ is an isomorphism (see~\cite[Definition~5.1.12]{Ishi}). We denote by $\mathrm{Cl}(X,x)$ the \emph{local divisor class group} of $(X,x)$ (see~\cite[Chapter~6]{NeBook}). The group $\mathrm{Cl}(X,x)$ can also be interpreted as the group of isomorphism classes of reflexive sheaves on $(X,x)$ of rank one \cite{Sakai}.

Let $M$ be a coherent $\Ss{X}$-module. Following~\cite{Gustavsen1} a \emph{connection} on $M$ is an $\Ss{X}$-linear map
\begin{equation*}
    \nabla : \mathrm{Der}_{\CC}(\Ss{X})\to \mathrm{End}_{\CC}(M),
\end{equation*}
which for all $f\in \Ss{X}$, $m\in M$ and $D\in \mathrm{Der}_{\CC}(\Ss{X})$ satisfies the \emph{Leibniz rule}
\begin{equation*}
    \nabla(D)(fm)=D(f)m+f \nabla(D)(m).
\end{equation*}
A connection $\nabla$ is called \emph{integrable or flat} if it is a $\CC$-Lie-algebra homomorphism.
Denote by:
\begin{align*}
    \mathrm{Ref}_{X} &:= \text{the category of reflexive $\Ss{X}$-modules},\\
    \mathrm{Ref}^{\nabla}_{X} &:= \text{the category of pairs $(M,\nabla)$ where $M$ is a reflexive $\Ss{X}$-module and $\nabla$ is integrable},\\
    \mathrm{Rep}_{\pi_1(\Sigma)} &:= \text{the category of complex finite dimensional representations of $\pi_1(\Sigma)$}.
\end{align*}
By~\cite{Gustavsen1} there is an equivalence
\begin{equation*}
    \mathrm{Ref}^{\nabla}_{X} \cong \mathrm{Rep}_{\pi_1(\Sigma)}.
\end{equation*}
In general, the {\it forgetful functor}
\begin{equation*}
    \mathrm{Ref}^{\nabla}_{X} \to \mathrm{Ref}_{X},
\end{equation*}
is not an equivalence of categories. Moreover, the forgetful functor may not be essentially surjective, i.e., the forgetful functor may not be onto on objects. This problem has an easier reformulation in the case of reflexive modules of rank one. Denote by
\begin{align*}
    \mathrm{Ref}^1_{X} &:= \text{the category of reflexive $\Ss{X}$-modules of rank one},\\
        \mathrm{Ref}^{1,\nabla}_{X} &:= \text{the category of pairs $(M,\nabla)$ where $M$ is a reflexive $\Ss{X}$-module of rank one}\\
        &\,\quad \,\text{and $\nabla$ is integrable},\\
    \mathrm{Rep}^1_{\pi_1(\Sigma)} &:= \text{the category of complex one dimensional representations of $\pi_1(\Sigma)$}.
\end{align*}

Let $\rho\in \mathrm{Obj}(\mathrm{Rep}^1_{\pi_1(\Sigma)})$. Thus, we have $\rho \colon \pi_1(\Sigma) \to \mathrm{G}L(1,\CC)=\CC^*$ where $\CC^*$ is the multiplicative subgroup of $\CC$.
Since $\CC^*$ is abelian,
\begin{equation}
\label{eq:property.abelian}
    \Hom(\pi_1(\Sigma),\CC^*) = \Hom(\pi_1(\Sigma)_{\mathrm{ab}},\CC^*) = \Hom(H_1(\Sigma,\ZZ),\CC^*).
\end{equation}
By the Universal Coefficient Theorem we get
\begin{equation}
\label{eq:universal.coef.th}
    \Hom(H_1(\Sigma,\ZZ),\CC^*) \cong H^1(\Sigma,\CC^*).
\end{equation}
Hence, by~\eqref{eq:property.abelian} and~\eqref{eq:universal.coef.th}
 $\rho \in H^1(\Sigma,\CC^*)$. Replacing $\CC^*$ by $\CC/\ZZ$,  the exact sequence
\begin{equation*}
    0 \to \ZZ \to \CC \to \CC/\ZZ \cong \CC^* \to 0,
\end{equation*}
induces  the following cohomological long exact sequence
\begin{equation}
\label{eq:long.sequence.coh.Sigma}
    \dots \to H^1(\Sigma, \CC) \to H^1(\Sigma, \CC/\ZZ) \stackrel{\Phi}{\longrightarrow}
     H^2(\Sigma, \ZZ) \to H^2(\Sigma, \CC) \to \dots
\end{equation}
By~\eqref{eq:long.sequence.coh.Sigma},  the image of $\Phi$ is torsion. Hence, it has a factorization
through
\begin{equation*}
    \Phi \colon  H^1(\Sigma, \CC/\ZZ) \to \mathrm{Tors}\left(H^2(\Sigma, \ZZ)\right)\cong H,
    \ \ \ \rho\mapsto\Phi(\rho)\in H.
\end{equation*}
On the other hand,
by the equivalence between $\mathrm{Ref}^{1,\nabla}_{X}$ and $\mathrm{Rep}^1_{\pi_1(\Sigma)}$, let $(M_{\rho},\nabla_{\rho})$ be the pair in $\mathrm{Obj}(\mathrm{Ref}^{1,\nabla}_{X})$ associated to $\rho$. Hence, in this reformulation,
the  {\it forgetful morphism} is
\begin{align*}
    \Psi \colon H^1(\Sigma, \CC/\ZZ) \to \mathrm{Cl}(X,x),\ \ \
    \rho =(M_\rho, \nabla_\rho)\mapsto \Psi(\rho)=M_\rho.
\end{align*}
Note also that the groups $H$ and $\mathrm{Cl}(X,x)$ are related by the following long exact sequence
\begin{equation}
\label{eq:def.c1.overline}
	\begin{gathered}
	\xymatrix{ 0 \ar[r]& H^1(\Sigma,\ZZ) \ar[r]& \CC^{p_g} \ar[r]& \mathrm{Cl}(X,x) \ar[r]^(0.65){\overline{c}_1}& H \ar[r]& 0,
	}
	\end{gathered}
\end{equation}
where $\overline{c}_1$ assigns to a Weil divisor the homological class ot its intersection with the link. In Appendix  we will prove that these maps fit together in a commutative diagram:
\begin{lemma}\label{lemma:Rep.H2}
 $\overline{c}_1\circ \Psi=\Phi$.\end{lemma}

\subsection{Full sheaves} Let $(X,x)$ be the germ of a normal surface singularity and
$ \pi \colon \Rs \to X $
be a resolution.
\begin{definition}
Let $\Sf{F}$ be a sheaf on $\Rs$. We say that $\Sf{F}$ is \emph{generically generated by global sections} if it is generated by global sections except in a finite set.
\end{definition}
Recall, the following definition of full sheaves from~\cite[Definition~1.1]{Ka}.
\begin{definition}
An $\Ss{\Rs}$-module $\Sf{M}$ is called \emph{full} if there is a reflexive $\Ss{X}$-module $M$ such that
$\Sf{M} \cong \left(\pi^* M\right)^{\smvee \smvee}$. We call $\Sf{M}$ the full sheaf associated to $M$.
\end{definition}
The following characterization of full sheaves will be very important in the following sections.
\begin{proposition}[{\cite[Proposition~1.2]{Ka}}]\label{fullcondiciones}
A locally free sheaf $\Sf{M}$ on $\Rs$ is full if and only if
\begin{enumerate}
\item the sheaf $\Sf{M}$ is generically generated by global sections.
\item The natural map $H^1_E(\Rs,\Sf{M}) \to H^1(\Rs,\Sf{M})$ is injective.
\end{enumerate}
If $\Sf{M}$ is the full sheaf associated to $M$, then $\pi_* \Sf{M}=M$.
\end{proposition}
In the particular case of a rational singularity the characterization of full sheaves is `easier':
\begin{proposition}[{\cite[Lemma~2.2]{Esnault}}]\label{Prop:fullcondiciones.rational}
Suppose that $(X,x)$ is a rational normal surface singularity. A locally free sheaf $\Sf{M}$ on $\Rs$ is full if and only if
\begin{enumerate}
\item the sheaf $\Sf{M}$ is generated by global sections,
\item $H^1_E(\Rs,\Sf{M}) =0$.
\end{enumerate}
\end{proposition}
\begin{definition}
Let $\Sf{F}$ be a sheaf on $\Rs$. The \emph{base points} of $\Sf{F}$ are the points $p \in \Rs$ such that $s(p)=0$ for all $s\in H^0(\Rs,\Sf{F})$. A component $E_v$ is called a \emph{fixed component} of $\Sf{F}$ if any section $s\in H^0(\Rs,\Sf{F})$ vanishes along $E_v$.
\end{definition}
In the particular case of  rank one sheaves, in Proposition \ref{fullcondiciones} the condition
 {\it (1)}, namely that
$\Sf{M}$ is generically generated by global sections, can equivalently be  replaced
by the condition that $\Sf{M}$ has no fixed components.
 This is proved in the following proposition.
\begin{proposition}
\label{prop:gggs.fixed}
Let $(X,x)$ be the germ of a normal surface singularity. Let $\pi \colon \Rs \to X$ be any resolution. Let $\Sf{L}$ be a line bundle such that the natural map $H^1_E(\Rs,\Sf{L}) \to H^1(\Rs,\Sf{L})$ is injective. Then, the following conditions are equivalent:
\begin{enumerate}
    \item the sheaf $\Sf{L}$ is generically generated by global sections.
    \item the sheaf $\Sf{L}$ has no fixed components.
\end{enumerate}
\end{proposition}
\begin{proof}
The implication {\it (1) $\Rightarrow$ (2)}  is immediate. The meaning of  {\it (2) $\Rightarrow$ (1)}
is the following: if $\Sf{L}$ has no fixed components (along $E$) then it cannot have any one-dimensional non-compact base component either. The proof runs as follows.
 Suppose that $\Sf{L}$ has no fixed components. Denote by
\begin{equation*}
    L:= \pi_* \Sf{L} \quad \text{and} \quad \tilde{\Sf{L}} := (\pi^* L)^{\smvee \smvee}.
\end{equation*}
Set $U:= X_{\mathrm{reg}}$, the regular part of $X$. Denote by $\iota \colon U \to X$ the inclusion.
 By hypothesis, the natural map  $H^1_E(\Rs,\Sf{L}) \to H^1(\Rs,\Sf{L})$ is injective,
 hence $H^0(\Rs,\Sf{L}) \to H^0(U,\Sf{L})$ is an isomorphism. This gives  isomorphisms
%
\begin{equation*}
    L_x \cong H^0(\Rs, \Sf{L}) \cong H^0(U, \Sf{L}) \cong (\iota_* \Sf{L}|_U)_x \cong (\iota_*\iota^*L)_x
\end{equation*}
Therefore, $L$ is a reflexive module (see~\cite[Proposition $1.6$]{stabhar}), and $\tilde{\Sf{L}}$ is a full sheaf. Since $\tilde{\Sf{L}}$ is generated by global sections on $\Rs \setminus E$ and the shaves $\Sf{L}$ and $\tilde{\Sf{L}}$ are isomorphic over $\Rs \setminus E$, we get
\begin{equation*}
    \{ p \in \Rs \, | \, \text{$\Sf{L}$ is not generated by global sections at $p$} \} \subset E.
\end{equation*}
Hence, since $\Sf{L}$ has no fixed components along $E$, it can have at most
 finitely many base points.
\end{proof}

In general, for a normal surface singularity the generation by global sections of a full sheaf depends on the resolution. Nevertheless, if for some resolution a full sheaf is generated by global sections then it satisfies nice properties under pull-back and push-forward. The following proposition will be used later.
\begin{proposition}
\label{Prop:Full.generation.global.sec}
 Let $(X,x)$ be the germ of a normal surface singularity. Let $\pi \colon \Rs \to X$ be any resolution. Let $\sigma \colon \Rs_0 \to \Rs$ be the blow-up in some point
  $p\in E\subset \Rs$. Let $\Sf{M}$ be a full $\Ss{\Rs}$-sheaf. If $\Sf{M}$ is generated by global sections, then $\sigma^* \Sf{M}$ is a full $\Ss{\Rs_0}$-sheaf generated by global sections. Moreover, we have the isomorphism $\sigma_* \sigma^* \Sf{M} \cong \Sf{M}$.
\end{proposition}
\begin{proof}
The proof is an adaptation of the proof of~\cite[Proposition~4.7]{BoRo}.
\end{proof}

\section{Rank one full sheaves of rational singularities}
\label{sec:rational.sing}
In this section we classify all the reflexive modules of rank one over any rational singularity.

Recall that if $(X,x)$ is rational then the link $\Sigma$ is a rational homology sphere,
in particular $H_1(\Sigma, \ZZ)={\rm Tors} (H_1(\Sigma, \ZZ))=L'/L=H$. Furthermore, since
${\rm Pic}^0(\Rs)=\CC^{p_g}=0$, we also have the isomorphism $c_1:{\rm Pic}(\Rs)\to L'$, that is,
all the line bundles are characterized topologically. In this section,
 for any $\ell'\in L'$ we denote
by $\Ss{\Rs}(\ell')\in {\rm Pic}(\Rs)$ that bundle which satisfies $c_1(\Ss{\Rs}(\ell'))=\ell'$.
(Later we will define the bundles $\Ss{\Rs}(\ell')$ in a more general  non-rational context as well.)

\begin{theorem}
\label{th:class.rational.case}
 Let $(X,x)$ be a rational normal surface singularity. Let $\pi\colon \Rs \to X$ be any resolution. Then,
 \begin{enumerate}
     \item for any $h\in H$ the line bundle $\Ss{\Rs}(-s_h)$ is full.
     \item If a line bundle $\mathcal{L}\in \mathrm{Pic}(\Rs)$ is full, then $c_1(\mathcal{L})=-s_h$ for some $h\in H$. %
          Hence $\mathcal{L}=\Ss{\Rs}(-s_h)$.
 \end{enumerate}
 Thus, any full sheaf of rank one is of the form $\Ss{\Rs}(-s_h)$ for some $h\in H$.
\end{theorem}
\begin{proof}
\textit{(1)} Let $h\in H$. In order to prove that the line bundle $\Ss{\Rs}(-s_h)$ is full,  by Proposition~\ref{Prop:fullcondiciones.rational} (and Serre duality) we have to verify  that
\begin{enumerate}[label=(\roman*)]
    \item the sheaf $\Ss{\Rs}(-s_h)$ is generated by global sections,
    \item $H^1(\Ss{\Rs}(s_h +K_{\Rs}))=0$.
\end{enumerate}
First, we verify  (i).  Since $c_1(\Ss{\Rs}(-s_h)) \in -\mathcal{S}'$ and $(X,x)$ is rational, then by~\cite[Theorem~12.1]{Lip} the sheaf $\Ss{\Rs}(-s_h)$ is generated by global sections.

\noindent
Next, we prove  (ii).
Since $\lfloor r_h \rfloor=0$, we get $r_h \in \{r_h\}-\mathcal{S}_{\QQ}$. Thus, by
Theorem~\ref{Th:GGR}
\begin{equation}
    \label{eq:Vanish.rh.0}
    h^1(\Ss{\Rs}(K_{\Rs}+r_h))=0.
\end{equation}
Set $\Delta:= s_h-r_h$. If $\Delta=0$, then we are done. Suppose that $\Delta >0$.
Then consider  the following exact sequence
\begin{equation}
\label{eq:Exact.Delta.1}
    0 \to \Ss{\Rs}(K_{\Rs}+r_h) \to \Ss{\Rs}(K_{\Rs}+s_h) \to \Ss{\Delta}(K_{\Rs}+s_h) \to 0.
\end{equation}
By~\eqref{eq:Vanish.rh.0} and the induced long  cohomological
exact sequence
we get
\begin{equation*}
     h^1(\Ss{\Rs}(K_{\Rs}+s_h)) = h^1(\Ss{\Delta}(K_{\Rs}+s_h)).
\end{equation*}
By Serre duality we have
\begin{equation*}
 h^1(\Ss{\Delta}(K_{\Rs}+s_h))=h^0(\Ss{\Delta}(-s_h+\Delta))=h^0(\Ss{\Delta}(-r_h)).
\end{equation*}
Therefore, we have to prove that $h^0(\Ss{\Delta}(-r_h))=0$. Now, the exact sequence
\begin{equation*}
    0 \to \Ss{\Rs}(-s_h) \to \Ss{\Rs}(-r_h) \to \Ss{\Delta}(-r_h) \to 0,
\end{equation*}
induces the cohomological  long exact sequence
\begin{equation}\label{eq:long.exact.2}
	\begin{gathered}
	\xymatrix{
	0 \ar[r]& H^0(\Ss{\Rs}(-s_h)) \ar[r]& H^0(\Ss{\Rs}(-r_h)) \ar[r]& H^0(\Ss{\Delta}(-r_h)) \ar@{->} `r/8pt[d] `/10pt[l] `^dl[lll]|{} `^d \\
	H^1(\Ss{\Rs}(-s_h)) \ar[r]& H^1(\Ss{\Rs}(-r_h)) \ar[r]& H^1(\Ss{\Delta}(-r_h)) \ar[r]& 0 }
	\end{gathered}
\end{equation}
Since $s(r_h)=s_h$, by Lemma~\ref{lemma:computation.seq} (taking $\mathcal{L}=\Ss{\Rs}(-r_h)$)  we get $H^0(\Ss{\Rs}(-s_h)) \stackrel{\cong}{\longrightarrow} H^0(\Ss{\Rs}(-r_h))$.
By Lipman's vanishing theorem~\ref{Th:Lipman.Vanish}, we get $h^1(\Ss{\Rs}(-s_h))=0$. Therefore, by~\eqref{eq:long.exact.2} we get
\begin{equation}
    \label{eq:Theo1.Equality.3}
    h^0(\Ss{\Delta}(-r_h))=0.
\end{equation}
This proves the $H^1(\Ss{\Rs}(s_h + K_{\Rs}))=0$. Therefore, $\Ss{\Rs}(-s_h)$ is a full sheaf.

\textit{(2)} Let $\mathcal{L}\in \mathrm{Pic}(\Rs)$ be a full sheaf. Now we prove that its first Chern class $c_1(\mathcal{L})$ is equal to $-s_h$ for some $h\in H$. Since $X$ is a rational singularity,  $\mathcal{L}$ is generated by global sections, hence $(c_1(\mathcal{L}),E_v)\geq 0$ for any $v\in V$. Thus, $\ell':=-c_1(\mathcal{L})\in \mathcal{S}'$.

Set $h:=[\ell'] \in H=L'/L$. By the properties of the  Lipman cone reviewed  in the preliminaries, there exists a unique minimal element $s_h\in {\mathcal S}_h'$ such that $[s_h]=h$. Therefore, $s_h \leq \ell'.$ Denote by $\delta := \ell'-s_h$.  If $\delta=0$, we are done. Indeed, $s_h=\ell'$ implies that $\mathcal{L}=\Ss{\Rs}(-s_h)$ (recall that $X$ is a rational singularity, thus $\mathrm{Pic}^0(\Rs)=0$). Suppose that $\delta>0$. Consider the exact sequence
\begin{equation}
\label{eq:ssq.1}
    0 \to \Ss{\Rs}(K_{\Rs}+s_h) \to \Ss{\Rs}(K_{\Rs}+\ell') \to \Ss{\delta}(K_{\Rs}+\ell') \to 0.
\end{equation}
Since $\mathcal{L}$ is full,  $H^1(\Ss{\Rs}(K_{\Rs}+\ell'))=0$. Thus, by the induced cohomological
long exact sequence 
\begin{equation*}
    H^1(\Ss{\delta}(K_{\Rs}+\ell'))=0.
\end{equation*}
By Serre duality
\begin{equation*}
    H^1(\Ss{\delta}(K_{\Rs}+\ell'))=H^0(\Ss{\delta}(-\ell'+\delta))= H^0(\Ss{\delta}(-s_h)),
\end{equation*}
hence
\begin{equation}
\label{eq:h0.sh}
    h^0(\Ss{\delta}(-s_h))=0.
\end{equation}
On the other hand, from the exact sequence
\begin{equation*}
    0 \to \Ss{\Rs}(-\ell') \to \Ss{\Rs}(-s_h) \to \Ss{\delta}(-s_h) \to 0,
\end{equation*}
and $h^1(\Ss{\Rs}(-s_h))=0$ by Lipman's vanishing theorem, we obtain
\begin{equation}
\label{eq:h1.sh}
    h^1(\Ss{\delta}(-s_h))=0.
\end{equation}
In particular, ~\eqref{eq:h0.sh} and~\eqref{eq:h1.sh} provide
\begin{equation*}
\chi(\Ss{\delta}(-s_h))=0.
\end{equation*}
This by Riemann-Roch theorem reads as
$0=\chi(\Ss{\delta}(-s_h))=\chi(\delta) - (\delta,s_h)$.
Since $(X,x)$ is   rational and $\delta>0$,  by Artin's criterion
$ \chi(\delta)> 0$.
Furthermore, since $s_h \in \mathcal{S}'$, we also obtain
 $- (\delta,s_h)\geq 0$.
Therefore, 
  $  0=\chi(\delta) - (\delta,s_h)>0$.
This is a contradiction. Therefore, $\delta$ must be equal to zero. 
\end{proof}
\begin{remark}
By definition, there is a one to one correspondence between full sheaves over $\Rs$ and reflexive modules over $X$. Since the set of  reflexive modules over $X$ is independent of the resolution,
it is natural to ask the independence of the resolution of
the classification of Theorem~\ref{th:class.rational.case}. Let us verify this fact in the next discussion.
\end{remark}
To verify the naturalness of the sheaves of the theorem  we need some notation. Let $(X,x)$ be a rational normal surface singularity. Let $\pi\colon \Rs \to X$ be any resolution. Let $\sigma \colon \Rs_0 \to \Rs$ the blow-up  at some point $p\in E\subset \Rs$.
Denote by
\begin{align*}
    L(\Rs)&:= H_2(\Rs,\ZZ), \quad &L'(\Rs)&:=H_2(\Rs, \partial \Rs, \ZZ)  \quad \text{and} \quad &H(\Rs)&:= L'(\Rs)/L(\Rs), \\
    L(\Rs_0)&:= H_2(\Rs_0,\ZZ), \quad &L'(\Rs_0)&:=H_2(\Rs_0, \partial \Rs_0, \ZZ) \quad \text{and} \quad &H(\Rs_0)&:= L'(\Rs_0)/L(\Rs_0).
\end{align*}
The map $\sigma$ induces, via the cohomological pullback and the duality $H_2(\Rs, \partial\Rs,\ZZ) \cong H^2(\Rs,\ZZ)$, the following maps
 $   \sigma^* \colon L'(\Rs) \to L'(\Rs_0)$ and $ \sigma^* \colon L(\Rs) \to L(\Rs_0).$
They induce an  isomorphism
  $  \sigma^* \colon H(\Rs) \to H(\Rs_0)$.
A direct  verification of the following statement (which basically follows from Theorem
~\ref{th:class.rational.case} as well)
is left to the reader.
\begin{proposition}
Consider the cycles $s_{h,\Rs}$ and $s_{h,\Rs_0}$ in the corresponding  resolutions. Set
 \begin{align*}
     \mathrm{Full}^1(\Rs)&:= \left \{ \Ss{\Rs}(-s_{h,\Rs})\in \mathrm{Pic}(\Rs) \text{ with $h\in H(\Rs)$} \right\},\\
     \mathrm{Full}^1(\Rs_0)&:= \left\{ \Ss{\Rs_0}(-s_{h,\Rs_0})\in \mathrm{Pic}(\Rs_0) \text{ with $h\in H(\Rs_0)$} \right\}.
 \end{align*}
 Then, the map $\sigma$ induces the bijection
  $   \sigma^*\colon \mathrm{Full}^1(\Rs) \to \mathrm{Full}^1(\Rs_0)$.
 Furthermore, the induced map by $\sigma$ satisfies
 \begin{equation*}
     \sigma^* \Ss{\Rs}(-s_{h,\Rs}) =  \Ss{\Rs_0}(-s_{\sigma^*(h),\Rs_0}).
 \end{equation*}
\end{proposition}

\begin{remark}
\label{remark:full.flat.rational}
Let $(X,x)$ be a rational normal surface singularity. By Theorem~\ref{th:class.rational.case} all the rank one full sheaves are classified by the group $H$. Since $H=H^2(\Sigma,\ZZ)$, where $\Sigma$ is the link of $(X,x)$, $H$ is a topological invariant. In particular, Theorem~\ref{th:class.rational.case} provides  a topological classification. Thus, the set of
 full sheaves of rank one only depends on the topology of $(X,x)$.
\end{remark}
The following theorem makes Remark~\ref{remark:full.flat.rational} even more precise.
\begin{theorem}\label{th:new}
Let $(X,x)$ be a rational normal surface singularity. Let $\pi\colon \Rs \to X$ be any resolution. Then, any full $\Ss{\Rs}$-sheaf of rank one is flat.
\end{theorem}
\begin{proof}
By Lemma~\ref{lemma:Rep.H2}, the following commutative diagram commutes
\begin{equation*}
    \begin{diagram}
         \mathrm{Cl}(X,x) &\rTo^{\overline{c}_1}& H\\
         \uTo^{\Psi} & & \uTo^=\\
         H^1(\Sigma, \CC/\ZZ) &\rTo^{\Phi}& H^2(\Sigma, \ZZ)
    \end{diagram}
\end{equation*}
Since $(X,x)$ is a rational singularity, by~\eqref{eq:long.sequence.coh.Sigma} and~\eqref{eq:def.c1.overline} the maps $\Phi$ and $\overline{c}_1$ are isomorphism. Hence, the map $\Psi$ is also an isomorphism.
\end{proof}

\begin{remark}
\label{remark:shcases}
(1) The computation of the set
$\{s_h\}_{h\in H}$, even for very concrete families of topological types
of singularities (families of links)
in general is not easy.
Recall that there exists an algorithm which provides $s_h$, cf. \ref{ss:GenLau},
however a close expression in general is missing. In special cases
 the known formulae are surprisingly technical and arithmetical.
For the cyclic quotient (string graphs) and star shaped graphs see \cite{Ne1}, for surgery
3-manifolds see \cite{Nesur1,Nesur2}. (All these are reported in \cite{NeBook} as well.)
 For general rational singularities a concrete close formula is not known (by the authors).

 (2) In \cite{Wu}  Wunram determined
 the first Chern classes of full sheaves associated with the minimal
 resolution of cyclic quotient singularities,  as cycles in $L'$. These  expressions can be
 identified with the expressions of $\{s_{h}\}_h$ given in \cite{Ne1}.
 We emphasize that in \cite{Wu} the universal property of this Chern class, as the minimal element of the Lipman cone with fixed $h\in H$, was not recognized. We believe that this new conceptual point of view brings essentially new perspectives in the theory
  (besides the generalization of full sheaves to the rational and minimally elliptic cases).

  (3) By the McKay's correspondence in the ADE case, non-trivial full shaves (of any rank)
are in bijection with the irreducible exceptional curves of the minimal resolution:
the Chern classes are of type $\{-E_v^*\}_{v\in V}$ (where each $-E_v^*$  can be represented, as a divisor, by a cut of $E_v$). The rank of $\Sf{M}$ with $c_1(\Sf{M})=-E_v^*$ is the $E_v$-multiplicity
$m_v$ of $Z_{min}$. Hence, the rank one non-trivial full sheaves  correspond to those components $E_v$ with $m_v=1$. On the other hand,
 for other,  more general singularities,  this correspondence between $\{E_v\}_v$ and
the full shaves is broken. (This happens in our rational case too:
$\{E_v\}_{v\in V}$ versus $\{s_h\}_{h\in H}$.) In order to keep (at least) part of this correspondence
$\{E_v\}_{v\in V}\, \leftrightarrow\, {\rm Ref}_X$, Wunram in \cite{Wu} introduced the family
of special full shaves (they will be discussed below).
\end{remark}

\begin{definition} Let $(X,x)$ be a rational
 normal surface singularity and $\Rs\to X$ a resolution.
A full sheaf $\Sf{M}\in{\rm Pic}(\Rs)$ is called {\it special} if $H^1(\Rs, \Sf{M}^*)=0$.
Then $\pi_*\Sf{M}$ is called a special reflexive module.
\end{definition}
For rational $(X,x)$ the trivial sheaf $\mathcal{O}_{\Rs}$ is full and special. Wunram in \cite[Theorem 1.2(b)]{Wu} proved the following fact.
\begin{proposition}\label{proposition:Wunram.quot}
Let $\Rs$ be the \underline{minimal} resolution of a rational  singularity. Then special
non-trivial  indecomposable reflexive modules (i.e. special full sheaves of $\Rs$) correspond bijectively with irreducible components
$\{E_v\}_{v\in V}$. By this correspondence $\Sf{M}_v\,\leftrightarrow\, E_v$, realized by
$c_1(\Sf{M}_v)=-E^*_v\in L'$,  one also has ${\rm rank}(\Sf{M}_v)=m_v$ (the $E_v$-coeffcient of $Z_{min}$).
\end{proposition} Next, we combine the statements of Theorem
\ref{th:class.rational.case} and  Proposition \ref{proposition:Wunram.quot}.

\begin{theorem}\label{theorem:special.rat}
Let $(X,x)$ be a rational singularity with a fixed resolution $\Rs$. Then, for some $s_h$ with $h\not=0$,  the following facts are equivalent:

(1) The full sheaf  $\Sf{L}={\mathcal O}_{\Rs}(-s_h)$ ($h\not=0$) is special,

(2) $(-s_h,Z_{min})=1$,

(3) $s_h=E^*_v$ for some $v\in V$ and $m_v=1$.
\end{theorem}
\begin{proof}  {\it (2)} and {\it (3)} are obviously equivalent. Indeed, write $s_h$ as
$\sum_v n_vE_v^*$ with certain $n_v \in \ZZ_{\geq 0}$.
Since $Z_{min}=\sum_vm_vE_v$ is supported on $E$ (i.e. $m_v\in\ZZ_{>0}$  for any $v$),
$(-s_h,Z_{min})=\sum_vn_vm_v$.

For the equivalence {\it (1)$\Leftrightarrow$(3)} we provide two proofs. The first one is analytic,
and basically it follows from Proposition \ref{proposition:Wunram.quot}, whenever the resolution is
minimal. For a non-minimal resolution one can prove the stability of the statements with respect to
a blow-up.

However, we present another topological/combinatorial proof as well, based on the combinatorics of
rational graphs and (generalized) Laufer sequences (valid in any resolution).
It shows some additional topological bridges, and
 it can serve as a prototype for further generalizations as well.

{\it The combinatorial proof of  (1)$\Rightarrow$(2)}. If ${\mathcal L}={\mathcal O}_{\Rs}(-s_h)$ is special then
$h^1(\mathcal L^*)=0$. Then we apply Lemma \ref{lemma:computation.seq}{\it (b)}
applied for   $c=-s_h$ and the bundle  ${\mathcal L}^*$.

Since $s_h\not=0$ and $(,)$ is negative definite,  there exists some $E_{v(0)}$ such that $(-s_h, E_{v(0)})>0$.

 Then,
consider the computation sequence $\{z_i\}_{i=1}^t$
which connects $E_{v(0)}$ with $Z_{min}$.  Write also  $z_0=0$.
Since $(X,x)$ is rational, by Laufer's rationality criterion,
$(z_i, E_{v_i})=1$  for all $0< i<t$.

Next, consider the sequence  $x_i:=-s_h+z_i$
for all such $i\geq 0$. It connects $-s_h$ with $-s_h+Z_{min}$. The point is that at each intermediate
step $(x_i, E_{v_i})=(-s_h, E_{v_i})+ (z_i, E_{v_i})\geq 1$ for $i>0$.
Therefore,  $\{x_i\}_{i\geq 0}$ can be considered as the beginning of a computation sequence,
which computes $h^1(\mathcal L^*)$ in
Lemma \ref{lemma:computation.seq}{\it (b)}. But since $h^1(\mathcal L^*)=0$, we must have
$\sum_i h^1({\mathcal O}_{{\mathbb P}^1}(-x_i, E_{v_i}))=0$ for this (partial) sequence,
hence $(x_i, E_{v_i})=1$ for all $i\geq 0$. For $i=0$ this reads as
$(-s_h, E_{v(0)})=1$ and for all other $i>0$ we must have $(-s_h, E_{v_i})=0$. Since $\sum_i
 E_{v_i}=Z_{min}$, summation over $i$
gives $(-s_h, Z_{min})=1$.

{\it The combinatorial proof  (3)$\Rightarrow$(1)}\   Consider again the computation sequence $\{x_i\}_i$ which connects $c=-s_h$ with
 $s(-s_h)$. Note that $x_0=-s_h$, and (since $s_h=E^*_v$) the next step  is necessarily
 $x_1=-s_h+E_v$, since $E_v$ is the only base element $E_u$ with $(-s_h,E_u)>0$. Using Lemma
 \ref{lemma:computation.seq}{\it (b)} $h^1({\mathcal O}_{\Rs}(s_h))=\sum_i ((x_i, E_{v_i})-1)$.
 Hence we have to show that $(x_i, E_{v_i})=1$ for any $i$.

As an independent construction,
 let $\Gamma^e_v$ be the `extended' graph constructed as follows. Let $\Gamma$ be the dual resolution graph of the resolution $\Rs\to X$. Then $\Gamma^e_v$ consists of $\Gamma$ and a  new vertex $v^e$,
  which is glued to $v$ (the vertex which appears in {\it (3)}) by an edge. Let $k$ be the Euler number of  $v^e$. One sees that if $k\ll 0$ then $\Gamma^e_v$ is negative definite. Furthermore,
  in such a case $k\ll 0$,  the $E_{v^e}$-multiplicity $m^e$
  of $E_{v^e}$ in $Z_{min}(\Gamma^e_v) $ is one (for details see e.g.  \cite[Th. 4.1.3]{GN}).
  On the other hand, by  \cite[Th. 4.1.3]{GN}, $\Gamma^e_v$ is rational if and only if $m_v$ (the $E_v$
  coefficient of $Z_{min}=Z_{min}(\Gamma)$) is one. Since this appears as an assumption in {\it (3)}
  we get that $\Gamma^e_v$ is rational. Let us consider a computation sequence for $Z_{min}(\Gamma^e_v)$.
  For the first step we choose $E_{v^e}$. Then we continue by the Laufer algorithm.
  (Hence the next added term is necessarily $E_v$.)
  In this way we obtain a series $\{y_j\}_j$. By the above discussion (since $k\ll 0$ and $m^e=1$), $E_{v^e}$ will be not chosen again.

  Let us compare the two computation sequences after we rewrite them as
  $x_i=-s_h+z_i=-E_v^*+z_i$ and $y_i=E_{v^e}+\bar{z}_i$. By comparing the two algorithms (and the terms) we see that we can make along the algorithms the common choices $z_i=\bar{z}_i$
  (and both sequences end simultaneously).
  Moreover, the terms $(x_i, E_{v_i})=(y_i,E_{v_i})$ are also   equal.
 Now let us repeat what we get. By  \cite[Th. 4.1.3]{GN}, $m_v=1$ implies that $\Gamma^e_v$ is rational.
 By Laufer criterion of rationality of this graph, $(y_i,E_{v_i})=1$ for every $i$, hence by the
 coincidence of the two sequences $(x_i,E_{v_i})=1$ too. Hence $h^1({\mathcal O}_{\Rs}(s_h))=0$ by
Lemma   \ref{lemma:computation.seq}{\it (b)}. Hence ${\mathcal O}_{\Rs}(-s_h)$ is special.

 [Though in the proof the next additional info is not visible, it might help the reader.
  Along the steps of the sequence $\{x_i\}_i$  and $\{y_i\}_i$
  we accumulate $\sum_iE_{v_i}= s(-s_h)+s_h= Z_{min}(\Gamma^e_v)-E_{v^e}$. In fact, it can be shown that this is a sum of Artin fundamental cycles of different support in $\Gamma$.  The first one is
 exactly $Z_{min}$, this first subsequence was considered in the proof of
{\it  (1)$\Rightarrow$(2)}. The tower of fundamental cycles is explicitly described in
 the proof of  \cite[Th. 4.1.3]{GN}. See also \cite{LT}.]
\end{proof}

Theorem \ref{theorem:special.rat} was formulated from the perspective of the Chern class $-s_h$. In the next statement  we reformulate it as  an equivalence from the perspective of the exceptional divisors of the resolution.
(Maybe it is worth to emphasize that for a graph  it can happen that
$s_{[E_v^*]}=s_{[E_u^*]}=E_v^*$ for some $v\not=u$, $v,u\in V$, see e.g. Example \ref{example:E*s}.)
 Note that the condition $s_h=E^*_v$ in part {\it (3)} says that $h=[E^*_v]$ and
$s_{[E^*_v]}=E^*_v$. Hence, {\it (3)} reads as $s_{[E^*_v]}=E^*_v$ and $m_v=1$ (with
 $h=[E^*_v]$). A non-obvious  point is that in the minimal resolution the condition $m_v=1$ by oneself already guarantees
 $s_{[E^*_v]}=E^*_v$. 

 \begin{corollary}\label{theorem:special.rat2}
Let $(X,x)$ be a rational singularity and let  $\Rs$ be its  \underline{minimal} resolution.
Then  the following facts are equivalent:

(1) $\Gamma^e_v$ is rational,

(2) $m_v=1$,

(3) $s_{[E^*_v]}=E^*_v$ and $m_v=1$,

(4) the  sheaf  $\Sf{L}={\mathcal O}_{\Rs}(-E^*_v)$  is non-trivial special full.

If any of these conditions hold then $[E^*_v]\not=0$.

Hence,  the  rank one non-trivial special full sheaves are classified in the minimal resolution by
vertices with $m_v=1$.
\end{corollary}
\begin{proof}
{\it (1)}$\Leftrightarrow${\it (2)} follows from  \cite[Th. 4.1.3]{GN} or \cite{LT}.
Let us verify that if any of the above conditions hold then $[E^*_v]\not=0$, hence
Theorem \ref{theorem:special.rat} can be applied. Indeed, assume e.g.  that $m_v=1$. Then
by Proposition \ref{proposition:Wunram.quot} we get that ${\mathcal O}_{\Rs}(-E^*_v)$ is full.
Assume that $[E^*_v]=0$, that is, $E^*_v\in L_{>0}$. Since ${\mathcal O}_{\Rs}(-E^*_v)$ is full,
$H^1(\Rs,{\mathcal O}_{\Rs}(K_{\Rs}+E^*_v))=0$, hence $H^1(E^*_v,
{\mathcal O}_{E^*_v}(K_{\Rs}+E^*_v))=0$ too, which by Serre duality implies $H^0({\mathcal O}_{E^*_v})=0$, which cannot hold since
$H^0({\mathcal O}_{E^*_v})$ contains at least the constants.

Hence, for all the cases we can assume that $[E^*_v]\not=0$.
Then
{\it (3)}$\Leftrightarrow${\it (4)} follows from  Theorem \ref{theorem:special.rat},
and  {\it (2)}$\Rightarrow${\it (4)} from
Proposition \ref{proposition:Wunram.quot} and  {\it (3)}$\Rightarrow${\it (2)} is obvious.
\end{proof}

It is interesting to compare (directly) {\it (1)} and {\it (4)}. Their equivalence relates
a (rational) surgery property of the link with special reflexive modules of the singularity.

In an arbitrary (non-minimal) resolution {\it (3)}$\Leftrightarrow${\it (4)} still holds
(since $[E^*_v]\not=0$ and Theorem \ref{theorem:special.rat} remain  valid). However  {\it (2)}$\Rightarrow${\it (3)}
and Proposition \ref{proposition:Wunram.quot}
do  not necessarily hold, see Example \ref{example:E*s}(e) when we create a redundant vertex with $m_v=1$. In non-minimal resolutions {\it (3)} is the right index set  for the classification.

\begin{example}\label{example:E*s}
Consider the following resolution graph $\Gamma$. By a computation $H=\ZZ_7$. The universal abelian covering
(with Galois group $H$) is the Poincar\'e sphere $\Sigma(2,3,5)$ with finite fundamental group,
hence $\Gamma$ is the graph of a quotient singularity.

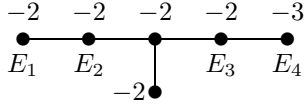
\begin{picture}(300,45)(0,0)
\put(50,35){\makebox(0,0){$-2$}}
\put(75,35){\makebox(0,0){$-2$}}
\put(100,35){\makebox(0,0){$-2$}}
\put(125,35){\makebox(0,0){$-2$}}
\put(150,35){\makebox(0,0){$-3$}}

\put(50,15){\makebox(0,0){$E_1$}}
\put(75,15){\makebox(0,0){$E_2$}}
\put(125,15){\makebox(0,0){$E_3$}}
\put(150,15){\makebox(0,0){$E_4$}}
\put(90,5){\makebox(0,0){$-2$}}
\put(50,25){\circle*{5}}
\put(75,25){\circle*{5}}
\put(100,25){\circle*{5}}
\put(125,25){\circle*{5}}
\put(150,25){\circle*{5}}
\put(100,5){\circle*{5}}
\put(50,25){\line(1,0){100}}
\put(100,5){\line(0,1){20}}
\end{picture}
\end{example}

We have the following cases regarding the different irreducible exceptional curves.

(a) $s_{[E_1^*]}=E^*_1$ and $m_1=1$, hence with Chern class $-E^*_1$ there exists a rank one
special full sheaf.

(b) $s_{[E_4^*]}=r_{[E_4^*]}=E^*_4$ and $m_{4}=1$, hence with Chern class $-E^*_4$ there exists a rank one
special full sheaf.

(c) $s_{[E_2^*]}=s_{[E_4^*]}=E^*_4$, hence $s_{[E_2^*]}\not =E^*_2$. In fact, $E_2^*$ does not
equal with any $s_h$ $(h\in H)$. Note also that $m_2=2$. Therefore,  with Chern class $-E^*_2$ there exists no  rank one full sheaf. However, with this Chern class there exists an indecomposable  rank 2 special full sheaf.

(d)  $s_{[E_3^*]}=E^*_3$ and $m_3=2$. Therefore,  with Chern class $-E^*_3$ there exists a  rank one non-special full sheaf, and  an indecomposable  rank 2 special full sheaf.
(Compare with \cite[Example~2]{Wu} as well.)

(e) Let us blow up $\Rs$ at a generic point of $E_4$, and let
$E_{new}$ be the new exceptional divisor. Then in the new graph $m_{new}=1$, however
$s_{[E_{new}^*]}=E^*_4\not=E^*_{new}$. That is, ${\mathcal O}_{\Rs}(-E^*_{new})$ is not full.

\section{Rank one full sheaves of minimally elliptic singularities}
\label{sec:min.ell}
Let  $(X,x)$ denote a minimally elliptic singularity.
Recall that in any resolution $Z_K\in L$. As usual, $C$ denotes the elliptic cycle.

First, we characterize the full shaves of rank one with non-trivial first Chern class.

\begin{theorem}
\label{th:char.full.minimally.ellip}
Let $(X,x)$ be a minimally elliptic singularity. Let $\pi \colon \Rs \to X$ be any resolution such that the support $|C|$ of $C$ is $E$ (this happens e.g. in the minimal resolution).
If $\Sf{L}$ is full sheaf of rank one such that $[c_1(\Sf{L})]$ is non-zero,
then $c_1(\Sf{L})=-s_h$ for some $h\in H$, $h\not=0$.
\end{theorem}
\begin{proof} The proof follows the same strategy as the proof of Theorem
\ref{th:class.rational.case}, however, in this case the package of results that we use from singularity theory
should be  valid for minimally elliptic germs, and even certain additional steps should be modified.

Let $\Sf{L}$ be a full sheaf of rank one such that $[c_1(\Sf{L})]$ is non-zero, hence $c_1(\Sf{L})$ is also different from zero. By Proposition~\ref{fullcondiciones}, the sheaf $\Sf{L}$ is generically generated by global sections.  Therefore, the sheaf does not have any fixed component. Thus, $\ell':=-c_1(\mathcal{L})\in \mathcal{S}'$. Moreover, since $|C|=E$ and $\ell'$ is a non-trivial element of $\mathcal{S}'$, by \cite[Theorem~7.2.31]{NeBook} we get $H^1(\Rs,\Sf{L})=0$. By Proposition~\ref{fullcondiciones}, the natural map from $H^1_E(\Rs,\Sf{L})$ to $H^1(\Rs,\Sf{L})$ is injective. But  $H^1(\Rs,\Sf{L})=0$, therefore
\begin{equation}
\label{eq:full.rank1.Ezero}
    H^1_E(\Rs,\Sf{L})=0.
\end{equation}

Set $h:=[\ell'] \in H=L'/L$. By the special  property of the Lipman cone (cf. preliminaries),
there exists a unique minimal element $s_h\in {\mathcal S}'$ such that $[s_h]=h$. Therefore, $s_h \leq \ell'$. Denote by $\delta := \ell'-s_h\in L_{\geq 0}$. We have to prove that $\delta=0$.  Suppose that $0<\delta$. Consider the exact sequence
\begin{equation}
\label{eq:exact.seq.L1}
    0 \to \Sf{L}^{\smvee}(K_{\Rs}) \otimes \Ss{\Rs}(-\delta) \to \Sf{L}^{\smvee}(K_{\Rs})  \to \Sf{L}^{\smvee}(K_{\Rs}) \otimes\Ss{\delta} \to 0.
\end{equation}
By Serre duality and~\eqref{eq:full.rank1.Ezero}
\begin{equation}
\label{eq:L.dual.zero}
    H^1(\Rs,\Sf{L}^{\smvee}(K_{\Rs}))\cong H^1_E(\Rs,\Sf{L})=0.
\end{equation}
By~\eqref{eq:L.dual.zero} and considering the long exact sequence of cohomology associated to~\eqref{eq:exact.seq.L1} we obtain
\begin{equation*}
H^1(\Rs,\Sf{L}^{\smvee}(K_{\Rs}) \otimes\Ss{\delta}) =0
\end{equation*}
By Serre duality again
\begin{equation}
\label{eq:delta.L.H0}
H^0(\delta,\Sf{L}(\delta)) =0.
\end{equation}
On the other hand, one can consider the following exact sequence as well:
\begin{equation}
\label{eq:exact.seq.delta1}
    0 \to \Sf{L} \to \Sf{L}(\delta) \to \Sf{L}(\delta)\otimes\Ss{\delta} \to 0.
\end{equation}
Since $-c_1(\Sf{L}(\delta))=\ell' -\delta=s_h \in \mathcal{S}'$,
 by~\cite[Theorem~7.2.31]{NeBook} we get $H^1(\Rs,\Sf{L}(\delta))=0$.
 Therefore, from  the cohomological long exact sequence
  associated with~\eqref{eq:exact.seq.delta1} and the previous vanishing, we get
\begin{equation}
\label{eq:delta.L.H1}
    H^1(\delta,\Sf{L}(\delta))=0.
\end{equation}
In particular, ~\eqref{eq:delta.L.H0} and~\eqref{eq:delta.L.H1} provide
\begin{equation}
\label{eq:char.delta.1.elliptic}
\chi(\Sf{L}(\delta)\otimes\Ss{\delta})=0.
\end{equation}
By Riemann-Roch theorem
\begin{equation}
\label{eq:char.delta.2.elliptic}
\chi(\Sf{L}(\delta)\otimes\Ss{\delta})=\chi(\delta) - (\delta,s_h)=0.
\end{equation}
Since $(X,x)$ is elliptic, we have $\min_{Z>0} \chi(Z)=0$. Therefore, $0 \leq \chi(\delta)$. Since $s_h \in \mathcal{S}'$, then $0\leq -(\delta,s_h)$. By the previous inequalities and by~\eqref{eq:char.delta.2.elliptic} we get
\begin{equation}
\label{eq:char.delta.2.elliptic1}
\chi(\delta)=0 \quad \text{and} \quad (\delta,s_h)=0.
\end{equation}
By assumption $\delta \neq 0$. Since $\chi(\delta)=0$, we get $C\leq \delta$. By hypothesis,  the support of $C$ is $E$. Therefore,  the support of $\delta$ is $E$, hence
 $(\delta,s_h)\neq 0$. This is a contradiction to the second equality of~\eqref{eq:char.delta.2.elliptic1}. Thus, $\delta=0$,  $-c_1(\Sf{L})=\ell'=s_h$, where $h=[-c_1(\Sf{L})]\in H$. This ends  the proof.
\end{proof}
Theorem~\ref{th:char.full.minimally.ellip} has a partial converse. First, we need the following propositions.
\begin{proposition}
\label{prop:Cohomology.Cond}
Let $(X,x)$ be a minimally elliptic singularity. Let $\pi \colon \Rs \to X$ be any resolution such that $|C|=E$. Let $h\in H$ be different from zero. If $\Sf{L}\in \mathrm{Pic}^{-s_h}(\Rs)$, then the natural map
\begin{equation*}
    H_E^1(\Rs,\Sf{L}) \to H^1(\Rs,\Sf{L}),
\end{equation*}
is injective.
\end{proposition}
\begin{proof}
Let $\Sf{L}\in \mathrm{Pic}^{-s_h}(\Rs)$. Since $-c_1(\Sf{L})\in \mathcal{S}'$, by~\cite[Theorem~7.2.31]{NeBook} we get
 \begin{equation}\label{eq:new}H^1(\Rs,\Sf{L})=0.
  \end{equation}
  In particular, we have to prove that $H_E^1(\Rs,\Sf{L})=0$ too.

   Set $\Delta:= s_h-r_h$ and   consider the exact sequence
\begin{equation}
\label{eq:exact.L.full1}
    0 \to \Sf{L}^{\smvee}(K_{\Rs}-\Delta) \to \Sf{L}^{\smvee}(K_{\Rs}) \to \Sf{L}^{\smvee}(K_{\Rs})\otimes \Ss{\Delta} \to 0.
\end{equation}
 By the Generalized Grauert–Riemenschneider Theorem~\ref{Th:GGR} we get
\begin{equation}
    \label{eq:Vanish.rh.1}
    h^1(\Sf{L}^{\smvee}(K_{\Rs}-\Delta))=0.
\end{equation}
The cohomological long exact sequence associated with ~\eqref{eq:exact.L.full1} and the vanishing~\eqref{eq:Vanish.rh.1} give
\begin{equation*}
    H^1(\Rs,\Sf{L}^{\smvee}(K_{\Rs})) \cong H^1(\Delta, \Sf{L}^{\smvee}(K_{\Rs})).
\end{equation*}
By Serre duality we get
\begin{equation*} H_E^1(\Rs,\Sf{L}) \cong H^1(\Rs,\Sf{L}^{\smvee}(K_{\Rs}))
     \quad \text{and} \quad
      H^1(\Delta, \Sf{L}^{\smvee}(K_{\Rs})) \cong H^0(\Delta, \Sf{L}(\Delta)).
\end{equation*}
Therefore, we have to prove that $h^0(\Delta, \Sf{L}(\Delta)) =0$. Consider the following exact sequence
\begin{equation*}
    0 \to \Sf{L} \to \Sf{L}(\Delta) \to  \Sf{L}(\Delta) \otimes \Ss{\Delta} \to 0.
\end{equation*}
By Lemma~\ref{lemma:computation.seq} (applying it to $\Sf{L}(\Delta)$), we get $H^0(\Rs,\Sf{L}) \cong H^0(\Rs,\Sf{L}(\Delta))$. Since $H^1(\Rs,\Sf{L})=0$,
cf. \eqref{eq:new},
 we get $h^0(\Delta,\Sf{L}(\Delta))=0$ too. This finishes the proof.
\end{proof}

\begin{proposition}
\label{prop:No.fixed.components}
Let $(X,x)$ be a minimally elliptic singularity. Let $\pi \colon \Rs \to X$ be any resolution such that $|C|=E$. Moreover, suppose that $E_v \cong \mathbb{P}_{\CC}^1$ for any $v \in V$.
If $\Sf{L}\in \mathrm{Pic}(\Rs)$ and $-c_1(\Sf{L}) \in \mathcal{S}'\setminus \{0\}$, then $\Sf{L}$ does not have fixed components.
\end{proposition}
\begin{proof}
The proof is analogous to the proof of ~\cite[Theorem~7.2.31.(b)]{NeBook}.
\end{proof}
The first part of the main classification result is the following theorem.
Its proof follows from  Proposition~\ref{fullcondiciones}, Proposition~\ref{prop:gggs.fixed}, Proposition~\ref{prop:Cohomology.Cond} and Proposition~\ref{prop:No.fixed.components}.
\begin{theorem}
\label{th:class.full.min}
Let $(X,x)$ be a minimally elliptic singularity. Let $\pi \colon \Rs \to X$ be any resolution such that $|C|=E$. Moreover, suppose that $E_v \cong \mathbb{P}_{\CC}^1$ for any $v \in V$.
Then,
\begin{equation*}
 \bigcup_{h\in H \setminus \{0\}}\mathrm{Pic}^{-s_h}(\Rs)
 = \{\mbox{rank one full sheaves with $[c_1({\mathcal L})]\not=0$}\}
 \subset \mathrm{Full}^1(\Rs).
\end{equation*}
\end{theorem}
 The missing case is $h=0$,  equivalently  $s_h=0$, cf. Remark~\ref{remark:1}.
 This case is done in the sequel.

\subsection{The case $h=0$}
In this section we study the families of full shaves with $[c_1(\mathcal{L})]=0$. This particular case should be treated via  a different strategy. First, we have the following characterization of full sheaves of rank one with trivial first Chern class.

\begin{proposition}
\label{prop:case.C1.is.0}
Let $(X,x)$ be the germ of a normal surface singularity. Let $\pi \colon \Rs \to X$ be any resolution. Let $\Sf{L}$ be a full sheaf of rank one. Then, the following statements are equivalent:
\begin{enumerate}
    \item the sheaf is trivial, i.e., $\Sf{L}\cong \Ss{\Rs}$,
    \item we have $c_1(\Sf{L})=0$.
\end{enumerate}
\end{proposition}
\begin{proof}
{\it  (1)$\Rightarrow$(2)} follows trivially. Next we prove {\it (2)$\Rightarrow$(1)}. Since $\Sf{L}$ is a full sheaf, it has  a non-trivial global section $s$ without fixed components. Since $c_1(\Sf{L})=0$, for any $p\in \Rs$ we get $s(p)\neq0$. Hence, the section $s$  trivializes  the line bundle $\Sf{L}$. 
\end{proof}
Consider the following example:
\begin{example}
\label{example1}
Consider the minimally elliptic singularity
\begin{equation*}
X=\{x^2+y^3+z^7=0\} \subset \CC^3.
\end{equation*}
By~\cite[p.~7]{Ne1} its link $\Sigma$ it is an integer homology sphere. Hence, $H=H^2(\Sigma,\ZZ)=0$. Denote by $\pi \colon \Rs \to X$ its minimal resolution. Since $H=0$, for any $\mathcal{L}\in \mathrm{Pic}(\Rs)$ we get $[c_1(\mathcal{L})]=0$. Nevertheless, by~\cite{DanRom} there exists non-trivial non-flat full sheaves of rank one.
\end{example}
By Proposition~\ref{prop:case.C1.is.0} we need to study the full sheaves with non-trivial first Chern class but with trivial class in $H$. By  Example~\ref{example1} such sheaves exist.

Let $(X,x)$ be a minimally elliptic singularity and $\pi \colon \Rs \to X$ any resolution such that $|C|=E$. Since $Z_{min}$ is integral, $\chi(Z_{min})=0$ and  $C$ is the minimally elliptic cycle, we have (see \cite{Ne,NeBook})
\begin{equation*}
    C=Z_{K} \leq Z_{min}.
\end{equation*}

\begin{theorem}
\label{th:class.h0}
(a) Let $(X,x)$ be a minimally elliptic singularity. Let $\pi \colon \Rs \to X$ be any resolution such that $|C|=E$. Let $\mathcal{L}$ be a full sheaf of rank one such that $[c_1(\mathcal{L})]=0$ and $c_1(\mathcal{L})\neq 0$. Then, $\mathcal{L} \in \mathrm{Pic}^{-Z_{min}}(\Rs)$.

(b)  Moreover, suppose that $E_v \cong \mathbb{P}_{\CC}^1$ for any $v \in V$.
Let $\mathcal{L}=\mathcal{G} \otimes \Ss{\Rs}(-Z_{min}) \in \mathrm{Pic}^{-Z_{min}}(\Rs)$ for some $\mathcal{G}\in \mathrm{Pic}^0(\Rs)$. Then,  $\mathcal{L}$  is a full sheaf if and only if $\mathcal{G}$ is not trivial.
\end{theorem}
\begin{proof}
{\it (a)}  Let $\mathcal{L}$ be a full sheaf of rank one such that $[c_1(\mathcal{L})]=0$ and $-c_1(\mathcal{L})\neq 0$. Since $\mathcal{L}$ does not have fixed components then $\ell':=-c_1(\mathcal{L}) \in \mathcal{S}'$. Since $[\ell']=0$ and $\ell' \neq 0$, then $\ell' \in \mathcal{S}_0 \setminus \{0\}$. Since $Z_{min} = \min\{\mathcal{S}_0\setminus\{0\}\}$, we get $Z_{min} \leq \ell'$. Now the proof follows exactly as the proof of Theorem~\ref{th:char.full.minimally.ellip}.

{\it (b)}
Let $\mathcal{L} \in \mathrm{Pic}^{-Z_{min}}(\Rs)$. By Proposition~\ref{fullcondiciones} we need to
test the following properties
\begin{enumerate}
    \item the sheaf $\mathcal{L}$ does not have fixed components,
    \item the natural map
 $   H^1_E(\Rs,\mathcal{L}) \to H^1(\Rs,\mathcal{L}) $
is an injection.
\end{enumerate}
(1) follows from Proposition~\ref{prop:No.fixed.components}. In order to test (2)
 write  $\mathcal{L}$ as $ \mathcal{G}(-Z_{min})$ for some $\mathcal{G}\in \mathrm{Pic}^0(\Rs)$.
 Then, by Serre duality,
\begin{equation*}
    H^1_E(\Rs,\mathcal{L})\cong H^1(\Rs,\mathcal{G}^{\smvee}(Z_{min}-Z_K)).
\end{equation*}
Suppose that $Z_{min}<Z_K$. Set $\delta=Z_{min}-Z_K > 0$. By~\cite[p.~26]{Ne}
(or by the dualized Generalized Grauert–Riemenschneider vanishing)
we get $h^0(\Ss{\delta}(\delta)\otimes \mathcal{G}^{\smvee})=0$. Thus, by the exact sequence
\begin{equation*}
    0 \to \mathcal{G}^{\smvee} \to \mathcal{G}^{\smvee}(\delta) \to \Ss{\delta}(\delta)\otimes \mathcal{G}^{\smvee} \to 0,
\end{equation*}
we get
\begin{equation*}
    0 \to H^1(\Rs,\mathcal{G}^{\smvee}) \to H^1(\Rs,\mathcal{G}^{\smvee}(\delta)) \to H^1(\Rs,\Ss{\delta}(\delta)\otimes \mathcal{G}^{\smvee}) \to 0.
\end{equation*}
Note that
\begin{equation*}
    H^1(\Rs,\Ss{\delta}(\delta)\otimes \mathcal{G}^{\smvee}) = -\chi(\Ss{\delta}(\delta)\otimes \mathcal{G}^{\smvee}).
\end{equation*}
By Riemann-Roch  and the ellipticity assumption $\chi(Z_{min})=0$,
\begin{align*}
    \chi\left(\Ss{\delta}(\delta)\otimes \mathcal{G}^{\smvee}\right)&=\chi(\Ss{\delta}(\delta))= \delta^2 + \chi(\delta)=\delta^2 -\frac{1}{2}\delta \cdot( \delta - Z_K)=\delta \cdot \left(\delta -\frac{1}{2}(\delta - Z_K)\right)\\
    &= \delta \cdot \left(\frac{1}{2}Z_{min}\right)=(Z_{min}-Z_K) \cdot \left(\frac{1}{2}Z_{min}\right)=-\chi(Z_{min})=0.
\end{align*}
Therefore,
\begin{equation*}
    H^1(\Rs,\mathcal{G}^{\smvee}) \cong H^1(\Rs,\mathcal{G}^{\smvee}(\delta)).
\end{equation*}
Clearly this isomorphism holds for $\delta=0$ too. Hence, by~\cite[Theorem~7.2.31]{NeBook}
\begin{equation}
\label{eq:case.2}
h^1(\mathcal{G}^{\smvee}(\delta))=h^1(\mathcal{G}^{\smvee}) =
\begin{cases}
0 \quad \text{if $\mathcal{G}^{\smvee} \neq \Ss{\Rs}$}\\
1 \quad \text{if $\mathcal{G}^{\smvee} \cong \Ss{\Rs}$}
\end{cases}
\end{equation}
Since $c_1(\mathcal{G}(-Z_{min}))=-Z_{min}$, we get $c_1(\mathcal{G}(-Z_{min})) \in -\mathcal{S}'$. Hence, by~\cite[Theorem~7.2.31]{NeBook}
\begin{equation*}
H^1(\Rs,\mathcal{G}(-Z_{min})) = 0.
\end{equation*}
This vanishing together with~\eqref{eq:case.2} shows that   the natural map
\begin{equation*}
    H^1_E(\Rs,\mathcal{G}(-Z_{min})) \to H^1(\Rs,\mathcal{G}(-Z_{min})),
\end{equation*}
is an injection if and only if $\mathcal{G}$ is not trivial. This finishes the proof.
\end{proof}
Thus, the classification statement  of rank one full sheaves
is the following.
\begin{theorem}
\label{theo:class.general}
Let $(X,x)$ be a minimally elliptic singularity. Let $\pi \colon \Rs \to X$ be any resolution such that $|C|=E$. Then
\begin{equation*}
    \mathrm{Full}^1(\Rs) \subset \left( \bigcup_{h\in H\setminus\{0\}}\mathrm{Pic}^{-s_h}(\Rs) \right) \cup \left ( \mathrm{Pic}^{-Z_{min}}(\Rs) \setminus \{\Ss{\Rs}(-Z_{min})\} \right) \cup \{\Ss{\Rs}\}.
\end{equation*}
Moreover, suppose that $E_v \cong \mathbb{P}_{\CC}^1$ for any $v \in V$. Then,
\begin{equation*}
    \mathrm{Full}^1(\Rs) = \left( \bigcup_{h\in H\setminus\{0\}}\mathrm{Pic}^{-s_h}(\Rs) \right) \cup \left ( \mathrm{Pic}^{-Z_{min}}(\Rs) \setminus \{\Ss{\Rs}(-Z_{min})\} \right)\cup \{\Ss{\Rs}\}.
\end{equation*}
\end{theorem}
\begin{remark}
\label{remark:Sing2.3.7}(a)
In the minimally elliptic case the classification of rank one full sheaves cannot be described only in terms of  first Chern classes.

(b)
Consider the minimally elliptic singularity $X=\{x^2+y^3+z^7=0\} \subset \CC^3$. In this case, $H=0$, hence by Theorem~\ref{theo:class.general}
\begin{equation*}
    \mathrm{Full}^1(\Rs) = \left ( \mathrm{Pic}^{-Z_{min}}(\Rs) \setminus \{\Ss{\Rs}(-Z_{min})\} \right) \cup \{\Ss{\Rs}\}.
\end{equation*}
Since $\Sigma$ is an integer homology $3$-sphere, we get $H^1(\Sigma,\CC/\ZZ)=0$. Therefore, the only one dimensional representation of $\pi_1(\Sigma)$ is the trivial one. Thus, $\Ss{\Rs}$ is the only flat full sheaf. The non-trivial non-flat full sheaves of rank one constructed in~\cite{DanRom}
belong to $\mathrm{Pic}^{-Z_{min}}(\Rs) \setminus \{\Ss{\Rs}(-Z_{min})\}$.
%
\end{remark}

\section{Flat and non-flat families of reflexive modules on minimally elliptic singularities}
\label{sec:flat.noflat}
In this section we classify those  reflexive modules on a minimally elliptic singularity that
admit flat connection. Since there is a one-to-one correspondence between reflexive modules and full sheaves (at some resolution), we study this question in the minimal resolution.

First we study the case when the link is a rational homology sphere.
In this case an interesting situation appears. Note that each ${\rm Pic}^{\ell'}(\Rs)$ is an affine
space, a ${\rm Pic}^{0}(\Rs)=\CC$ torsor. Hence the space of rank one full sheaves consists of one dimensional families. On the other hand, $H_1(\Sigma,\ZZ)$ is finite, hence the set of flat bundles is discrete: for any $h\in H$ we have to choose exactly one concrete bundle from
${\rm Pic}^{-s_h}(\Rs)$, which is flat. In the $h=0$ case this is easy, it is the trivial bundle
${\mathcal O}_{\Rs}$. For $h\not=0$ the choice should be done by a precise principle.

In fact, such a choice   is already present in the literature, under the name {\it natural
line bundles}, cf. \cite{Gr,NeBook}. If $\Rs$ is the resolution of a normal surface singularity
with rational homology sphere link then the morphism $c_1: {\rm Pic}(\Rs)\to L'$ has a section
(morphism of groups). For any $\ell'\in L'$ there exists  a unique line bundle ${\mathcal L}_{\ell'}
\in {\rm Pic}^{\ell'}(\Rs)$ with the following universal property:
if $N\ell'$ is an integral cycle for some $N\in \ZZ_{>0}$ then $({\mathcal L}_{\ell'})^{\otimes N}={\mathcal O}_{\Rs}(N\ell')$.
This line bundle will be denoted by ${\mathcal O}_{\Rs}(\ell')$.

\begin{theorem}
\label{theo:class.1}
Let $(X,x)$ be a minimally elliptic singularity such that its link is a rational homology sphere. Let $\pi \colon \Rs \to X$ be the minimal resolution.

Corresponding to $h=[c_1({\mathcal L})]=0$, among the rank one full sheaves
\begin{equation*}
    \left ( \mathrm{Pic}^{-Z_{min}}(\Rs) \setminus \{\Ss{\Rs}(-Z_{min})\} \right) \cup \{\Ss{\Rs}\}
\end{equation*}
only the trivial sheaf $\Ss{\Rs}$ admits a flat connection.

Corresponding to $h=[c_1({\mathcal L})]\not=0$, among the rank one full sheaves ${\rm Pic}^{-s_h}(\Rs)$
only  ${\mathcal O}_{\Rs}(-s_h)$ admits a flat connection.
\end{theorem}
\begin{proof} The case $h=0$ is clear. In the sequel assume that
$h\not=0$. By Theorem~\ref{th:class.full.min} any element of
 $\mathrm{Pic}^{-s_h}(\Rs)=(\bar{c}_1)^{-1}(h)$ is a full sheaf. By Lemma~\ref{lemma:Rep.H2}, the following diagram commutes
\begin{equation*}
    \begin{diagram}
         0  &\rTo& H^1(\Sigma,\ZZ) &\rTo& H^1(\Rs,\Ss{\Rs}) &\rTo& \mathrm{Cl}(X,0) &\rTo^{\overline{c}_1}& H &\rTo& 0\\
           && \uTo^= &&  && \uTo^{\Psi} &\circlearrowleft& \uTo^= && \\
           && H^1(\Sigma,\ZZ) &\rTo& H^1(\Sigma,\CC) &\rTo& H^1(\Sigma,\CC/\ZZ) &\rTo^{\Phi}& H &\rTo& 0
    \end{diagram}
\end{equation*}
Here  $H^1(\Sigma,\ZZ)=0$ and $H^1(\Rs,{\mathcal O}_{\Rs})\simeq \CC$.
Furthermore,
 $H^1(\Sigma,\CC)=0$ too, hence $\Phi$ is an isomorphism. Hence the only flat bundle in
 $(\bar{c}_1)^{-1}(h)\simeq\CC$ is $\Psi(\Phi^{-1}(h))$.

Let us write $\rho_h \in H^1(\Sigma,\CC/\ZZ)$ for the representation with $\Phi(\rho_h)=h$.
Since $H$ is torsion, there exists $N \in \mathbb{N}$ such that $N h=0$.
 Denote $\rho_h \otimes \rho_h \otimes \dots \otimes \rho_h$ (the tensor of $\rho_h$ with itself $N$-times) by
$\rho_h^{\otimes N}$.
Since $\Phi\left(\rho_h^{\otimes N}\right)=N h = 0$ and $\Phi$ is an isomorphism, we get that $\rho_h^{\otimes N}$ is isomorphic to the trivial representation. Therefore, $\Psi\left(\rho_h^{\otimes  N}\right)=\Ss{X}$. Now,
set  $\mathcal{L}_h:=\Psi(\rho_h)$. Therefore, $\mathcal{L}^{\otimes N}_h$ agrees with $\Psi(\rho_h^{\otimes N})$ in $\Rs \setminus E$. Therefore, there exists an integral  cycle $\ell\in L$ such that
 $   \mathcal{L}^{\otimes N}_h\cong \Ss{\Rs}(\ell)$.
This finishes the proof.
\end{proof}
Finally we treat the  cusp singularities. They are defined by the property that their
 minimal resolution graph is a cyclic (loop) with all $g(E_v)=0$ \cite{Laufer77,Ne}.
\begin{theorem}
\label{theo:class.2}
Let $(X,x)$ be a cusp singularity. Let $\pi \colon \Rs \to X$ be the minimal resolution.
Then, any reflexive module of rank one admits a flat connection.
\end{theorem}
\begin{proof}
 Let $L$ be a reflexive module of rank one and let $\Sf{L}$ be the associated full sheaf. By Theorem~\ref{theo:class.general} we have two cases:

\vspace{1mm}

\noindent \underline{{\bf Case}\ $\Sf{L}\in \mathrm{Pic}^{-Z_{min}}(\Rs) \setminus \{\Ss{\Rs}(-Z_{min})\}$.}
 \ By assumption, there exists $\mathcal{G}_L\in \mathrm{Pic}^{0}(\Rs)\setminus \{0\}$ such that
 $\Sf{L} =    \mathcal{G}_L \otimes \Ss{\Rs}(-Z_{min}).$
Hence
\begin{equation*}
\Sf{L}|_U \cong    \mathcal{G}_L|_U.
\end{equation*}
Therefore, in this case it is enough to prove that any element in $\mathrm{Pic}^0(\Rs)$ admits a flat connection. Next we prove this fact.
Consider the following commutative triangle,
\begin{equation*}
    \begin{diagram}
         H^1(\Rs,\ZZ) &&\\
         \dTo &\rdTo&\\
         H^1(\Rs,\CC)&\rTo&H^1(\Rs,\Ss{\Rs})
    \end{diagram}
\end{equation*}
The map from $H^1(\Rs,\ZZ)$ to $H^1(\Rs,\Ss{\Rs})$ is an injection, thus the map from $H^1(\Rs,\CC)$ to $H^1(\Rs,\Ss{\Rs})$ is not the zero map. Now, since $h^1(\Rs,\CC)=h^1(\Rs,\Ss{\Rs})=1$ we get
that the natural map
 $   H^1(\Rs,\CC) \to  H^1(\Rs,\Ss{\Rs})$ is an isomorphism.
This provides the next diagram:
\begin{equation}
\label{eq:diag.flat.2}
    \begin{diagram}
         0  &\rTo& H^1(\Rs,\ZZ) &\rTo& H^1(\Rs,\Ss{\Rs}) &\rTo& \mathrm{Pic}^0(\Rs) &\rTo& 0\\
           && \uTo^= && \uTo^{\cong}  &&  && && \\
           0  &\rTo& H^1(\Rs,\ZZ) &\rTo& H^1(\Rs,\CC) &\rTo^{\gamma}& H^1(\Rs,\CC/\ZZ) && \\
    \end{diagram}
\end{equation}
By~\eqref{eq:diag.flat.2}, we get $\mathrm{Pic}^0(\Rs) \cong \mathrm{Image}(\gamma)$.
Therefore, $\mathrm{Pic}^0(\Rs) \subset H^1(\Rs,\CC/\ZZ)=\mathrm{Rep}^1_{\pi_1(\Rs)}$
This proves this case.

\vspace{1mm}

\noindent
\underline{{\bf Case} \  $\Sf{L} \not\in \mathrm{Pic}^{-Z_{min}}(\Rs) \setminus \{\Ss{\Rs}(-Z_{min})\}$.}  \
 If $L$ is the trivial bundle, we are done. Suppose that $L$ is not trivial. By Theorem~\ref{theo:class.general}, there exists $h \in H$ different from zero such that $\Sf{L} \in \mathrm{Pic}^{-s_h}(\Rs)$. By Lemma~\ref{lemma:Rep.H2}, the following diagram commutes
\begin{equation*}
    \begin{diagram}
         0  &\rTo& H^1(\Sigma,\ZZ) &\rTo& H^1(\Rs,\Ss{\Rs}) &\rTo& \mathrm{Cl}(X,0) &\rTo^{\overline{c}_1}& H &\rTo& 0\\
           && \uTo^= &&  && \uTo^{\Psi} &\circlearrowleft& \uTo^= && \\
           && H^1(\Sigma,\ZZ) &\rTo& H^1(\Sigma,\CC) &\rTo^{\gamma}& H^1(\Sigma,\CC/\ZZ) &\rTo^{\Phi}& H &\rTo& 0
    \end{diagram}
\end{equation*}
Let $\rho$ be any element in $H^1(\Sigma, \CC/\ZZ)$ such that $\Phi(\rho)=h$. Denote by $L_\rho = \Psi(\rho)$ and $\Sf{L}_\rho= (\pi^*L_\rho)^{\smvee \smvee}$. The sheaf $\Sf{L}_\rho$ is a full sheaf such that $[c_1(\Sf{L}_\rho)]=h$, therefore $\Sf{L}_\rho \in \mathrm{Pic}^{-s_h}(\Rs)$. Since $\Sf{L}$ and $\Sf{L}_\rho $ have the same first Chern class, there exists $\mathcal{G} \in \mathrm{Pic}^{0}(\Rs)$ such that
 $   \Sf{L} \cong \Sf{L}_\rho\otimes \mathcal{G}$.
Therefore,
\begin{equation}
\label{eq:flat.4}
    \Sf{L}|_U \cong \Sf{L}_\rho|_U\otimes \mathcal{G}|_U.
\end{equation}
By the first case, $\mathcal{G}$ admits a flat connection. By construction, $\Sf{L_\rho}|_U$ admits a flat connection (recall that $L_\rho$ comes from the  representation $\rho$). Thus, by~\eqref{eq:flat.4} $\Sf{L}|_U$ admits a flat connection. This finishes the proof.
\end{proof}
\begin{remark} In fact, we proved that  the natural map
 $   \mathrm{Ref}^{1,\nabla}_{X} \to \mathrm{Ref}^1_{X}$
is a surjection. Indeed, let us denote by
   $ \mathrm{Vect}^1_{U}$ the category of line bundles over $U$,   and by
$    \mathrm{Vect}^{1,\nabla}_{U}$  the category of flat line bundles over $U$,
 where   $U= X_{\text{reg}}$. Moreover, consider the following diagram
\begin{equation}
\label{eq:dia.flat.1}
    \begin{diagram}
         \mathrm{Ref}^{1,\nabla}_{X} &\rTo& \mathrm{Ref}^1_{X}\\
         \dTo && \dTo\\
         \mathrm{Vect}^{1,\nabla}_{U} &\rTo& \mathrm{Vect}^1_{U}
    \end{diagram}
    \end{equation}
By the Riemann-Hilbert corresponence (see~\cite[Lemma~3.4]{Gustavsen1}) we get $\mathrm{Ref}^{1,\nabla}_{X}\cong \mathrm{Vect}^{1,\nabla}_{U}$. By~\cite[Corollary~3.12]{BuDr}, the inclusion $\iota\colon U \to X$ induces the isomorphism $\mathrm{Ref}^1_{X}\cong \mathrm{Vect}^1_{U}$.
On the other hand, we just proved previously that
 $\mathrm{Vect}^{1,\nabla}_{U} \to \mathrm{Vect}^1_{U}$ is a surjection. Hence the statement follows from
\eqref{eq:dia.flat.1}.
\end{remark}

\appendix
\section{}
\label{sec:proof}
\begin{lemma}
The following commutative diagram commutes
\begin{equation*}
    \begin{diagram}
         \mathrm{Cl}(X,x) &\rTo^{\overline{c}_1}& H\\
         \uTo^{\Psi} && \uTo^=\\
         H^1(\Sigma, \CC/\ZZ) &\rTo^{\Phi}& H
    \end{diagram}
\end{equation*}
\end{lemma}
\begin{proof}
Let $\rho \in  H^1(\Sigma, \CC/\ZZ)$.  Hence, $\rho$ is a one-dimensional representation of $\pi_1(\Sigma)$. Denote by $L_\rho := \Psi(\rho)$. Let $\pi\colon \Rs \to X$ be any resolution. Denote by $\Sf{L}_\rho := \left (\pi^* L_\rho \right)^{\smvee \smvee}$.  Set
\begin{equation*}
h:= \Phi(\rho) \quad \text{and} \quad h':=[c_1(\Sf{L}_\rho)]=\overline{c}_1(\Psi(\rho)).
\end{equation*}
Denote by
\begin{align*}
    U&:=X_{\text{reg}}=\Rs \setminus E,\\
    \mathrm{Pic}^{\text{top}}(\Rs) &:= \text{the topological Picard group of $\Rs$},\\
    \mathrm{Pic}^{\text{top}}(U) &:= \text{the topological Picard group of $U$},\\
    \mathrm{Pic}^{\text{top},\nabla}(U) &:= \text{the topological Picard group of $U$ of flat shaves}.
\end{align*}
Moreover, denote by
\begin{align*}
    \mathrm{top} \colon \mathrm{Pic}(\Rs) &\to \mathrm{Pic}^{\text{top}}(\Rs)\\
    \Sf{L} &\mapsto \Sf{L}^{\mathrm{top}}
\end{align*}
the natural map. Consider the following commutative diagram
\begin{equation*}
\begin{diagram}
     \mathrm{Pic}^{\text{top},\nabla}(U) &\rTo& H^1(U,\CC/\ZZ)\\
     \dTo&\circlearrowleft&\dTo^{\Phi}\\
     \mathrm{Pic}^{\text{top}}(U) &\rTo^{c_1^{\text{top}}}& \mathrm{Tors} \left(H^2(U,\ZZ)\right)
\end{diagram}
\end{equation*}
where $c_1^{\text{top}}$ is the topological Chern class. Recall that the topological first Chern class of a flat vector bundle is always torsion, hence the previous  diagram is commutative. Recall that $\Sigma$ is a deformation retract of $U$, thus
\begin{equation*}
  \mathrm{Tors} \left(H^2(U,\ZZ)\right) \cong  \mathrm{Tors} \left(H^2(\Sigma,\ZZ)\right) \cong H \quad \text{and} \quad H^1(U,\CC/\ZZ) \cong  H^1(\Sigma,\CC/\ZZ).
\end{equation*}
Hence,
\begin{equation}
\label{eq:top.rep.1}
    c_1^{\text{top}}\left ( (\Sf{L}_{\rho}|_U)^{\mathrm{top}} \right) = \Phi(\rho).
\end{equation}
Now, by the following commutative diagram
\begin{equation*}
\begin{diagram}
     \mathrm{Pic}^{\text{top}}(\Rs) &\rTo& \mathrm{Pic}^{\text{top}}(U)  \\
     \dTo^{c_1^{\text{top}}}&\circlearrowleft&\dTo^{c_1^{\text{top}}}\\
     H^2(\Rs,\ZZ) &\rTo& H^2(U,\ZZ)
\end{diagram}
\end{equation*}
we have
\begin{equation}
\label{eq:top.rep.2}
c_1^{\text{top}}\left ( (\Sf{L}_{\rho}|_U)^{\mathrm{top}} \right) = [c_1^{\text{top}} \left( \Sf{L}_{\rho}^{\mathrm{top}} \right)].
\end{equation}
By the following commutative diagram
\begin{equation*}
\begin{diagram}
     \mathrm{Pic}(\Rs) &\rTo& \mathrm{Pic}^{\text{top}}(\Rs)  \\
     \dTo^{c_1}&\circlearrowleft&\dTo^{c_1^{\text{top}}}\\
     H^2(\Rs,\ZZ) &\rTo& H^2(\Rs,\ZZ)
\end{diagram}
\end{equation*}
we have
\begin{equation}
\label{eq:top.rep.3}
c_1^{\text{top}} \left( \Sf{L}_{\rho}^{\mathrm{top}} \right)=c_1\left( \Sf{L}_{\rho} \right).
\end{equation}
By the equalities~\eqref{eq:top.rep.1},~\eqref{eq:top.rep.2} and~\eqref{eq:top.rep.3} we get
\begin{equation*}
h'=[c_1\left( \Sf{L}_{\rho} \right)]=\Phi(\rho)=h.
\end{equation*}
This finishes the proof.
\end{proof}

 \end{document}